\documentclass[a4paper,12pt]{amsart}
\usepackage{amsmath, amscd, amssymb, mathrsfs, mathtools,fullpage}
\usepackage[all]{xy}
\usepackage[l2tabu, orthodox]{nag}
\usepackage{hyperref}

\title{Topological polylogarithms and $p$-adic interpolation of $L$-values of totally real fields}

\author{Alexander Beilinson}
\address[Beilinson]{Dept. of Mathematics, University of Chicago, Chicago, IL 60637,
USA}
\email{sasha@math.uchicago.edu}

\author{Guido Kings}
\address[Kings]{Fak. f\"ur Mathematik \\
Universit\"at Regensburg\\
93040 Regensburg\\
Germany}
\email{guido.kings@ur.de}

\author{Andrey Levin}
\address[Levin]{National Research University HSE,
7 Vavilova Str., Moscow,  Russia}
\email{alevin57@gmail.com}
\thanks{The authors research was partially supported by the following grants:  NSF grant DMS-1406734 (Beilinson), SFB 1085 Higher invariants (Kings)}

\theoremstyle{plain}
    \newtheorem{theorem}{Theorem}[subsection]
    \newtheorem{lemma}[theorem]{Lemma}
    \newtheorem{proposition}[theorem]{Proposition}
    \newtheorem{corollary}[theorem]{Corollary}
\theoremstyle{definition}
    \newtheorem{definition}[theorem]{Definition}

\theoremstyle{remark}
    \newtheorem{remark}[theorem]{Remark}

\DeclareMathOperator{\TSym}{TSym}
\DeclareMathOperator{\Sym}{Sym}
\DeclareMathOperator{\Hom}{Hom}

\DeclareMathOperator{\Aut}{Aut}
\DeclareMathOperator{\Spec}{Spec}
\DeclareMathOperator{\Spf}{Spf}

\DeclareMathOperator{\GL}{GL}
\DeclareMathOperator{\PGL}{PGL}

\DeclareMathOperator{\gr}{gr}
\DeclareMathOperator{\pr}{pr}

\DeclareMathOperator{\mom}{mom}

\DeclareMathOperator{\Log}{Log}
\DeclareMathOperator{\pol}{pol}
\DeclareMathOperator{\Tr}{Tr}
\DeclareMathOperator{\N}{N}

\DeclareMathOperator{\vol}{vol}
\DeclareMathOperator{\ev}{ev}
\DeclareMathOperator{\Meas}{Meas}
\DeclareMathOperator{\restr}{restr}
\DeclareMathOperator{\contr}{contr}
\DeclareMathOperator{\mult}{mult}

\newcommand{\can}{\mathrm{can}}
\newcommand{\res}{\mathrm{res}}
\newcommand{\cont}{\mathrm{cont}}
\newcommand{\prolim}{\varprojlim}

\newcommand{\sLog}{{\sL og}}
\newcommand{\fra}{\mathfrak{a}}
\newcommand{\frb}{\mathfrak{b}}
\newcommand{\frc}{\mathfrak{c}}

  \newcommand{\frf}{\mathfrak{f}}

\newcommand{\sC}{\mathscr{C}}

\newcommand{\sF}{\mathscr{F}}
\newcommand{\sG}{\mathscr{G}}

\newcommand{\sK}{\mathscr{K}}
\newcommand{\sL}{\mathscr{L}}

\newcommand{\sP}{\mathscr{P}}

\newcommand{\cA}{\mathcal{A}}

\newcommand{\cO}{\mathcal{O}}

\newcommand{\cU}{\mathcal{U}}

\newcommand{\CC}{\mathbf{C}}
\newcommand{\HH}{\mathbf{H}}
\newcommand{\QQ}{\mathbf{Q}}
\newcommand{\RR}{\mathbf{R}}
\newcommand{\ZZ}{\mathbf{Z}}

\newcommand{\bbH}{\mathbb{H}}

\newcommand{\bfone}{\mathbf{1}}

\newcommand{\Zp}{{\ZZ_p}}
\newcommand{\Gm}{\mathbf{G}_m}

    \DeclareFontFamily{U}{wncy}{}
    \DeclareFontShape{U}{wncy}{m}{n}{<->wncyr10}{}
    \DeclareSymbolFont{mcy}{U}{wncy}{m}{n}
    \DeclareMathSymbol{\Sha}{\mathord}{mcy}{"58}

\newcommand{\Eis}{\mathrm{Eis}}

\newcommand{\tors}{\mathrm{tors}}

\newcommand{\isom}{\cong}
\newcommand{\sgn}{\mathrm{sgn}}
\newcommand{\id}{\mathrm{id}}

\numberwithin{equation}{subsection}

\begin{document}

\begin{abstract}
We develop the topological polylogarithm which provides an integral version
of Nori's Eisenstein cohomology classes for $\GL_n(\ZZ)$ and
yields  classes with values in an Iwasawa algebra. This implies directly the integrality properties of
special values of $L$-functions of totally real fields and a construction of the associated $p$-adic $L$-function.
Using a result of Graf, we also apply this to prove some
integrality and $p$-adic interpolation results for
the Eisenstein cohomology
of Hilbert modular varieties.
\end{abstract}

\maketitle

\tableofcontents

\section*{Introduction}
In the beginning of the 90s Nori \cite{Nori} and Sczech \cite{Sczech} almost simultaneously and independently developed the so called 
Eisenstein cohomology classes for $\GL_n(\ZZ)$ with rational 
coefficients and showed that one can get the Klingen-Siegel
theorem, about the rationality of zeta values of totally real
fields at negative integers, as a direct consequence. 

The approach by Nori involves the de Rham complex and is therefore 
restricted to rational coefficients. Sczech's construction is
analytic in nature and he gets rational cohomology classes in
the end by studying Dedekind sums. 

In this paper we present a different approach, depending on the topological polylogarithm, which is very much inspired by Nori's beautiful construction, but works with almost arbitrary 
coefficients. Moreover, the cohomology class we construct
has values in the formal completion of the group ring 
of a  finitely generated free abelian group and is hence exactly the Iwasawa algebra if one considers 
$p$-adic coefficents.

The main idea of our construction can be explained as follows.
Nori's cohomology classes should be really considered not as classes
on the locally symmetric space associated to $\GL_n(\ZZ)$ but rather
on the universal family of topologicial metrized tori above it. 
On this universal family these classes are completely determined by a residue condition,
so that the comparison map between de Rham and singular cohomology
in Nori's approach becomes unnecessary. In particular, Nori's construction interpreted in this way
works for almost arbitrary coefficients.

From our construction we get directly the integrality results
of Deligne-Ribet and the $p$-adic
interpolation of the special $L$-values of totally real fields. 
In fact we get directly cohomology classes with values in
the Iwasawa algebra. 
We also
explain a new result, building upon results of Graf \cite{Graf},
about the integrality and $p$-adic interpolation of Eisenstein cohomology classes for Hilbert modular varieties.

In recent years the question of finding integral versions of
Sczech's Eisenstein cocycle or of Shintani's
construction received considerable interest. 
Charollois and Dasgupta were able to refine Sczech's construction
to the integral level in \cite{Charollois-Dasgupta}. But because
of problems with their smoothing construction they could only
prove part of the integrality result of Deligne-Ribet for the
$L$-values. 
Another approach to $p$-adic interpolation is via the Shintani
cocycle of Hill \cite{Hill} and Solomon, which was refined to
give $p$-adic interpolation by Spiess \cite{Spiess-Shintani-cocycles} and Steele \cite{Steele} independently. 

Our approach is completely different from all the above and 
relies only on the cohomological properties of the so called
logarithm sheaf. Being purely topological, we do not need
to choose any extra data as for the Shintani cocycle (which
depends on a Shintani decomposition) nor do we have any severe 
restrictions on the coefficients as our approach is not 
analytic at all.

Moreover, in a geometric situation, where one replaces the tori 
by abelian varieties, a completely parallel story  for
other cohomology theories (and even for
motivic cohomology) can be developed. This leads to $p$-adic 
interpolation of motivic cohomology classes and hence of 
non-critical $L$-values (see \cite{Kings-Eisenstein} and
the applications of this theory in \cite{KLZ1}). This gives 
a further argument for pursuing the approach by the topological
polylogarithm.

The essential results on $p$-adic interpolation of $L$-values of this paper were obtained many years ago 
in the nineties  but were never 
published. The only exception is the case of the Riemann zeta 
function which is covered in \cite{BeilinsonLevin}.
The newer results 
on the Eisenstein classes of Hilbert modular 
varieties and their
$p$-adic interpolation depends on the purely topological construction
of the Harder's Eisenstein cohomology in the thesis of Graf (forthcoming \cite{Graf}).

\subsection*{Acknowledgements}
The first author would like to thank M. Nori for the introduction
to his Eisenstein cohomology classes back in 1992. 
The second author would like to thank the University of Chicago
for a very profitable stay in 2002 where he became first 
acquainted with the topological polylogarithm. He also would like to 
thank M. Nori for discussions at that time about
the possibility to construct Harder's Eisenstein classes for
Hilbert modular varieties with his $\GL_n(\ZZ)$ cohomology classes.

\section{The topological polylogarithm}
\subsection{Group rings of lattices}
We consider free abelian groups $L$ of finite rank $n$, which we call \emph{lattices}.
Let $A$ be a commutative ring and  $A[L]$ the group ring of $L$ with coefficients in $A$. We write
$\delta:L\to A[L]^\times$, $\ell\mapsto \delta_\ell$ for the universal group homomorphism.
In particular, $\ell\in L$ acts on $A[L]$ by multiplication with $\delta_\ell$. Let $L_A:=A\otimes_\ZZ L$. 

\begin{definition}
The completion of  $A[L]$ with respect to the augmentation ideal $J$ is denoted by
\[
R:=R(L):=\varprojlim_k A[L]/J^k.
\]
We write $I:=JR$ and consider $R$ with the filtration defined by the $I^kR$ and the induced $L$-action
$\delta:L\to R^\times$. We write $R^{(k)}:=R/I^{k+1}$ and
$R_A$ if we need to express the dependence on $A$.
\end{definition}
\begin{remark}
Let $T(L^\vee):=\Spec A[L]=\underline{\Hom}(L,\Gm)$ be 
the algebraic torus with character group $L$ over $\Spec A$. 
The augmentation ideal $J:=\ker(A[L]\to A)$ defines the
zero section of the smooth map $T(L^\vee)\to \Spec A$
and hence is a regular ideal. Note also that it is stable under the $L$-action. Then
$\Spf R$ is the formal group associated to $T(L^\vee)$.
In particular, if $\ell_1,\ldots,\ell_n$ is a basis of $L$, then $R$ is a power series
ring in the $\delta_{\ell_1}-1,\ldots,\delta_{\ell_n}-1$. 
\end{remark}
\begin{lemma}\label{lemma:gr-hom} There is
an isomorphism 
\begin{equation}
\Sym^\cdot L_A\isom \gr^\cdot_I R=\bigoplus_{k\ge 0} I^k/I^{k+1}.
\end{equation}
The induced action of $L$ on $\gr^\cdot_I R$ is trivial.
\end{lemma}
\begin{proof}
As $L$ is abelian we have an isomorphism $L_A\isom J/J^2\isom I/I^2$, which sends $1\otimes \ell$ to
$\delta_\ell-1\bmod J^2$. As $J$ and hence $I$ is a regular
ideal, the induced map $\Sym^\cdot L_A\to \gr^\cdot_I R$
is an isomorphism. If $a\in I^k$ then $(\delta_\ell-1)\cdot a\equiv 0\bmod I^{k+1}$, so that $\delta_\ell\cdot a\equiv a\bmod I^{k+1}$, which implies that $L$ acts trivially on $\gr^\cdot_I R$.
\end{proof}
The formation of $R$ is functorial in $L$: For each homomorphism $\varphi:L\to L'$ we have an $A$-algebra homomorphism
\[
\varphi_R:R\to R', 
\]
where $R':=R(L')$,  
which respects the filtrations by $I$ and $I'$.
\begin{definition}
A homomorphism of lattices $\varphi:L\to L'$ is called an \emph{isogeny}, if it is injective with finite cokernel. 
For an isogeny $\varphi$ we denote by $\deg \varphi:=\#(L'/\varphi(L))$  the \emph{degree} of $\varphi$. 
\end{definition}
\begin{proposition}\label{prop:gr-hom}
Let $\varphi:L\to L'$ be an isogeny with $\deg\varphi$ invertible
in $A$, then 
\[
\varphi_R:R\to R'
\]
is an isomorphism.
\end{proposition}
\begin{proof}
Both rings $R$, $R'$ are complete and separated so that it
suffices to show that 
\[
\gr_I\varphi_R:\Sym L_A\to \Sym L'_A
\]
is an isomorphism.
The $A$-module $I^k/I^{k+1}$ is generated by products of elements of the form $(\delta_\ell-1)^r$
and $\varphi_R$ maps these to $(\delta_{\varphi(\ell)}-1)^r$. This shows that $\gr_I\varphi_R=\Sym \varphi_A$, where
$\varphi_A:L_A\to L'_A$ is the induced map. If $\deg\varphi$ is invertible in $A$, $\varphi_A$ and hence $\gr_I\varphi$ 
is an isomorphism.
\end{proof}
For the closer investigation of $R$ we need the completion
of the divided power algebra of $L_A$.
\begin{definition}
Let $\Gamma L_A=\bigoplus_{k\ge 0}\Gamma_kL_A$ be the graded 
divided power algebra of $L_A$. For
$\ell\in L_A$ we write $\ell^{[k]}$ for the 
$k$-th divided power of $\ell$ and write 
\[
\widehat{\Gamma}L_A:=\prolim_r \Gamma L_A/I^{[r]}
\]
where $I:=\Gamma_+L_A$ is the augmentation ideal.  
\end{definition}
Note that one has $\ell^n=n!\ell^{[n]}$ and the formula
\begin{align}
(\ell+\ell')^{[k]}=\sum_{m=0}^{k}\ell^{[m]}\ell'^{[k-m]}
\end{align} 
in $\Gamma L_A$. As $\Gamma(L_A\oplus L_A)\isom \Gamma L_A\otimes_A\Gamma L_A$ the algebra  
$\Gamma L_A$ is a graded Hopf algebra. Its (graded)
dual is $(\Gamma L_A)^*\isom \Sym L_A^*$, where $L_A^*$ is the
$A$-dual of $L_A$. As $L_A$ is free one has a canonical 
isomorphism
\begin{equation}
\Gamma L_A\isom \TSym L_A
\end{equation}
with the Hopf algebra of the symmetric tensors. The isomorphism
$L_A\isom \Gamma_1L_A$ induces an $A$-algebra homomorphism 
\begin{equation}\label{eq:sym-hom}
\widehat{\Sym} L_A\to \widehat{\Gamma}L_A,
\end{equation}
where $\widehat{\Sym} L_A$ is the completion of $\Sym L_A$ at
its augmentation ideal.
Explicitly, if we choose a basis $\ell_1,\ldots,\ell_n$ of $L_A$,
this homomorphism is given by
\begin{equation}
\ell_1^{k_1}\cdots \ell_n^{k_n}\mapsto 
k_1!\cdots k_n!\ell_1^{[k_1]}\cdots \ell_n^{[k_n]}.
\end{equation}
From this description it is clear that \eqref{eq:sym-hom} is
an isomorphism if $A$ is a $\QQ$-algebra.
\begin{lemma}\label{lemma:exponential-map}
There is an $A$-algebra homomorphism
\[
\exp^*:R\to \widehat{\Gamma}L_A
\]
compatible with the filtrations. We write $\exp^*_k:R\to \Gamma_kL_A$
for the composition with the projection to $\Gamma_kL_A$.
\end{lemma}
\begin{proof}
Consider the group homomorphism 
$L\to (\widehat{\Gamma}L_A)^\times$ given by 
$\ell\mapsto \sum_{k\ge 0}\ell^{[k]}$.
This induces an $A$-algebra homomorphism $A[L]\to \widehat{\Gamma}L_A$ which maps $(\delta_\ell-1)^r$ into $(\widehat{\Gamma}_+L_A)^{[r]}$
and hence
 $J^r$ to $(\widehat{\Gamma}_+L_A)^{[r]}$. This induces 
the desired $A$-algebra homomorphism
$\exp^*:R\to \widehat{\Gamma}L_A$.
\end{proof}

\begin{remark}\label{rem:mom}
The map $\exp^*$ is induced from the exponential map
of the formal group $\Spf R$. For this one should think of
$\Spf \widehat{\Gamma}L_A$ as the divided power formal neighbourhood
of $0$ in the Lie algebra $L_A^*$ of $\Spf R$. The homomorphism
$\exp^*$ has also the following description.
 Let
$\bbH:=\varinjlim_r\Hom_{A}(R/I^r, A)$
be the bigebra of translation invariant differential operators
on $\Spf R$. Then one has $R\isom \Hom_A(\bbH,A)$ and
one has a map $\cU(L_A^*)\to \bbH$ of the universal enveloping 
algebra of the Lie algebra $L_A^*$ to $\bbH$. If we observe that 
$\cU(L_A^\vee)\isom \Sym L_A^*$ we get an $A$-algebra homomorphism 
\[
R\isom \Hom_A(\bbH,A)\to \Hom_A(\cU(L_A^\vee),A)\isom \widehat{\Gamma} L_A
\]
which coincides with $\exp^*$. 
\end{remark}
\begin{proposition} 
If $A$ is a $\QQ$-algebra, then 
\[
\exp^*:R\to \widehat{\Gamma}L_A
\]
is an isomorphism.
\end{proposition}
\begin{proof}

Identify $\gr_IR\isom \Sym I/I^2\isom\Sym L_A$. Then we claim that
the associated 
graded of $\exp^*:R\to \widehat{\Gamma}L_A$
\[
\gr_I\exp^*:\Sym L_A\to \Gamma L_A
\]
coincides with the canonical map. But 
$L_A\isom I/I^2$ is generated by $\delta_\ell-1$ which maps to
$(\sum_{k \ge 0}\ell^{[k]})-1$. This is congruent to 
$\ell^{[1]}$ modulo $({\Gamma}_+L_A)^{[2]}$. Thus 
$\gr_I\exp^*$ must be induced from the isomorphism $L_A\isom \Gamma_1L_A$ by the universal property of $\Sym L_A$.
In particular, if $A$ is a $\QQ$-algebra then
$\gr_I\exp^*$ is an isomorphism. As $R$ and $\widehat{\Gamma}L_A$
are complete and separated, this implies that $\exp^*$ is an isomorphism.
\end{proof}
Usually it is more convenient to work with the power
series ring $\widehat{\Sym}L_A$ than with $\widehat{\Gamma}L_A$.
\begin{corollary}
Let $A$ be a $\QQ$-algebra, then the isomorphism 
\[
\exp^*:R\isom \widehat{\Gamma}L_A\isom \widehat{\Sym} L_A
\]
is induced by the group homomorphism $\exp:L\to \widehat{\Sym} L_A$
which maps $\ell\mapsto\sum_{k\ge 0}\frac{\ell^{\otimes k}}{k!}$.
\end{corollary}
\begin{proof}
This is clear as $\frac{\ell^{\otimes k}}{k!}\mapsto \ell^{[k]}$
under $\widehat{\Sym} L_A\isom \widehat{\Gamma} L_A$.
\end{proof}
\subsection{Iwasawa algebras of lattices}
This section is not needed for the construction of the
topological polylogarithm, but it is needed later in the construction
of the $p$-adic measures.

Fix a prime number $p$. 
In this section $A$ will be a  $p$-adically complete and
separated ring. 
\begin{definition}
The \emph{Iwasawa algebra} $A[[L_{\Zp}]]$ is the completed
group ring
\[
A[[L_{\Zp}]]:=\varprojlim_r A[L/p^rL]
\]
where the projective limit is taken with respect to $A[L/p^{r+1}L]\to A[L/p^rL]$.
\end{definition}
The $A$-algebra $R$ is canonically isomorphic to  the Iwasawa algebra.
\begin{proposition}\label{prop:R-as-Iwasawa}
The map $\delta:L\to R^\times$ induces a continuous $A$-algebra
isomorphism
\[
A[[L_\Zp]]\xrightarrow{\isom} R.
\] 
\end{proposition}
\begin{proof}
Consider
the composition
$L\xrightarrow{\delta} A[L]^\times \to (A[L]/(p,J)^{r+1})^\times$.
By induction on $r$ one sees that $\delta_{p^r\ell}-1=
\delta_\ell^{p^r}-1\in (p,I)^{r+1}$. This implies that
this composition factors through $L/p^rL$ and one gets by the
universal property of the group ring an $A$-algebra homomorphism
$A[L/p^rL]\to A[L]/(p,J)^{r+1}$, such that the composition
$A[L]\to A[L/p^rL]\to A[L]/(p,J)^{r+1}$ is the quotient map.
This induces a continuous homomorphism
\[
A[[L_{\Zp}]]\to \varprojlim_r A[L]/(p,J)^{r+1}
\]
which is an isomorphism on the subring $A[L]$.
As $A[L]$ is a dense subring on
both sides and
both rings $A[[L_{\Zp}]]$ 
and $ \varprojlim_r A[L]/(p,J)^{r+1}$ are 
complete and separated, the homomorphism itself must be an isomorphism. It remains to show that 
\[
R\isom  \varprojlim_r A[L]/(p,J)^{r+1},
\] 
i.e., that $R$ is $(p,I)$-adically complete
and separated. As $(p,I)^{2r}\subset (p)^r+I^r\subset (p,I)^r$
the $(p,I)$-adic topology on the finitely generated 
$A$-module $R/I^r\isom A[L]/J^r$ coincides
with the $(p)$-adic one. Hence the $R/I^r$ are complete
in the $(p,I)$-adic topology, so that also $R$ is $(p,I)$-adically
complete. As $\bigcap_{r\ge 0}(p)^r=0$ and $\bigcap_{r\ge 0}I^r=0$
it is also separated.
\end{proof}
\begin{definition}\label{def:moment-map}
Let $A$ be a $p$-adically complete and separated ring. Then we
call 
\[
\mom:A[[L_\Zp]]\xrightarrow{\isom} R\xrightarrow{\exp^*}\widehat{\Gamma} L_A\isom \widehat{\TSym}L_A
\]
the \emph{moment map}. The projection onto its $k$-th component
\[
\mom^k:A[[L_\Zp]]\xrightarrow{\isom} R\to \TSym^kL_A
\]
we call the \emph{$k$-th moment map}. 
\end{definition}
To explain the name "moment map" recall that $A[[L_A]]$ can
be interpreted as the algebra of measures on $L_{\Zp}$.
\begin{definition}
Let $\sC(L_{\Zp})$ be the continuous $A$-valued functions on
$L_{\Zp}$. An $A$-valued measure is an $A$-linear map
$\mu:\sC(L_{Zp})\to A$. We write 
\[
\Meas(L_{\Zp},A):=\Hom_A(\sC(L_{\Zp}),A)
\]
for the space of all $A$-valued measures.
\end{definition}
It is well-known that $\Meas(L_{\Zp},A)$ is a ring under
convolution of measures which is canonically isomorphic to $A[[L_{\Zp}]]$.
\begin{proposition}\label{prop:mom-as-moment-map}
Identify $\Meas(L_{\Zp},A)\isom A[[L_{\Zp}]]$ and
let 
\[
\Meas(L_{\Zp},A)\isom R\xrightarrow{\mom^k}{\TSym}^kL_{\Zp}
\]
be the composition of the isomorphism in Proposition \ref{prop:R-as-Iwasawa} with the $k$-th moment map. 
If we interpret the $A$-dual 
$(\TSym L_{A})^*\isom \Sym L_A^*$ as polynomial functions 
$x_1^{k_1}\cdots x_n^{k_n}$ on $L_{A}$, 
then 
\[
\mom^k(\mu)=\sum_{k_1+\ldots+k_n=k}\mu(x_1^{k_1}\cdots x_n^{k_n}) 
\ell_1^{[k_1]}\cdots \ell_n^{[k_n]}
\]
where
$\mu(x_1^{k_1}\cdots x_n^{k_n})$ are the moments of the measure
$\mu$.
\end{proposition}
The proposition  follows by a direct calculation, as we do not
need it, we skip the proof.
\subsection{Torsors and locally constant sheaves}
We follow the principle
``right action on spaces, left action on cohomology''.

Let $G$ be a group and $\pi:X\to S$ be a right $G$-torsor. For a left $G$-module $M$ we define a $G$ action on $X\times M$
by $(x,m)g:=xg,g^{-1}m)$ and write as usual
\[
X\times^GM:=X\times M/G
\]
for the orbits of $G$ on  $X\times M$.
\begin{definition}
For a left $G$-module $M$, we define the locally constant sheaf
$\widetilde{M}$ to be the sheaf of sections of $X\times^GM$ over $S$ (where $M$ has the discrete topology). 
If the $L$-action is trivial then $\widetilde{M}$ is 
the constant sheaf 
$\underline{M}$.
\end{definition} 
The sections over $U\subset S$ open of the sheaf
$\widetilde{M}$ are explicitly given by 
\begin{equation}\label{eq:ass-sheaf}
\widetilde{M}(U)=\{f:\pi^{-1}(U)\to M\mid f(ug)= g^{-1} f(u)\mbox{ for all }g\in G, u\in \pi^{-1}(U)\}.
\end{equation}
If $X$ is simply connected, then the functor 
\begin{align}
\begin{split}
\{G-\mbox{modules}\}&\to \{\mbox{locally constant sheaves on }S\}\\
M&\mapsto \widetilde{M}
\end{split}
\end{align}
is an equivalence of categories. The inverse functor is 
$\sF\mapsto \Gamma(X,\pi^*\sF)$.
%
We apply this in the case of lattices.
\begin{definition}
Let $L$ be  a lattice. We write $V:=\RR\otimes L$ where
$\ell\in L$ acts from the right on $V$ by $v\mapsto v+\ell$.
We denote by 
\[
T:=T(L):=V/L
\]
the associated compact real torus.
\end{definition}
Over $T$
we have the fundamental $L$-torsor $V$
\begin{equation}\label{eq:fundamentaltorsor}
0\to L\to V\xrightarrow{\pi} T\to 0
\end{equation}
with $\pi^{-1}(0)=L$. 
\begin{definition}\label{def:torsor} Let $R^\times_1:=(1+I)^\times\subset R^\times$ be the subgroup of 
$1$ units. 
The $R^\times_1$-torsor $\Log^\times$
on $T$ is the 
push-out of the sequence \eqref{eq:fundamentaltorsor} with
$\delta:L\to R^\times_1$, so that one has an exact sequence of abelian groups
\begin{equation}
0\to R^\times_1 \to \Log^\times\xrightarrow{\pr_1} T\to 0.
\end{equation}
\end{definition}
Note that we  also have $\Log^\times :=V\times^LR^\times_1$.
The $R^\times_1$-torsor $\Log^\times$ is obviously 
rigidified over $0\in T$ by $1\in R_1^\times$. By  \cite[Expose VII, Proposition 1.3.5]{SGA7I} the group structure on
$\Log^\times$ can be uniquely recovered from its $R_1^\times$-torsor
structure together with its rigidification $1$ of its fibre 
$\Log_0^\times$ in $0\in T$. 

%
\subsection{The logarithm sheaf}
%
We will consider local systems on
the compact torus
\[
T:=T(L):=V/L.
\]
\begin{proposition}
There exists a local system $\sLog=\sLog_T$ on  $T$ of free
rank one $\underline{R}$-modules,  such that the monodromy action
 $L$-action
$L\to \Aut(0^*\sLog)= R^\times$ coincides with $\delta:L\to R^\times$. Let $\bfone\in 0^*\sLog$ be a
generator, then the pair $(\sLog,\bfone)$ is unique up to unique isomorphism.
\end{proposition}
\begin{proof}
Uniqueness: Let $(\sL,s)$ be another pair with the properties of $\sLog$. Then there exists a unique 
$L$-equivariant isomorphism
$\alpha:0^*\sLog\isom 0^*\sL$ with $\alpha(\bfone)=s$. 
Hence there is a unique isomorphism of local systems $\sLog\isom \sL$. 

Existence: We give two constructions. For the first consider $R$ as $L$-module via $\delta:L\to R^\times$ and
define
$\sLog:=\widetilde{R}$.
As generator $\bfone \in 0^*\sLog=R$ we choose the element $1\in R$.

For the second let $\pi_!\underline{A}$ 
be the direct image with compact supports of
the constant sheaf $\underline{A}$ on $V$. The sheaf $\pi_!\underline{A}$ is a local system of $\underline{A[L]}$-modules of rank one 
and $0^*\pi_!\underline{A}=A[L]$ has $1\in A[L]$ as generator. Hence we can take
\begin{equation}
\sLog:=\underline{R}\otimes_{\underline{A[L]}}\pi_!\underline{A}
\end{equation}
with the induced generator $\bfone\in 0^*\sLog$.
\end{proof}

\begin{definition}
We call $(\sLog,\bfone)$ the \emph{logarithm sheaf} and we 
let 
\begin{equation*}\label{def:Log-def}
\Log:=V\times^LR
\end{equation*}
so that $\sLog$ is the sheaf of sections of $\Log$.
\end{definition}
\begin{proposition} \label{prop:log-properties}
The logarithm sheaf $(\sLog,\bfone)$  has the following properties.
\begin{enumerate}
\item Consider the filtration  $I^k\sLog:=\widetilde{I^k}$ on $\sLog$.
Then there is a unique identification of local systems of 
$\gr_I^\cdot R=\Sym L_A$ modules  
\begin{equation*}\label{eq:gr-log}
\gr_I^\cdot\sLog\isom {\Sym}^\cdot \underline{L_A}
\end{equation*}
that maps $\bfone\bmod I\sLog$ to $1 \in \Sym^0 L_A=A$.
\item Let $\varphi:L\to L'$ be a homomorphism of lattices
and $\varphi:T\to T'$ be the induced map, then one has an
homomorphism of local systems
\[
\varphi_{\sLog}:\sLog_{T}\to \varphi^*\sLog_{T'},
\]
which is compatible with the filtrations and respects the generators $\bfone,\bfone'$.
\item If $\varphi:L\to L'$ is an isogeny and $\deg\varphi$ invertible in $A$, then 
\[
\varphi_{\sLog}:\sLog_{T}\to \varphi^*\sLog_{T'},
\]
is an isomorphism.
\item Let $+:T\times T\to T$ be the group structure on the torus, then
one has a unique isomorphism
\[
\pr_1^*\sLog\otimes_{\underline{R}}pr_2^*\sLog\isom +^*\sLog ,
\]
under which $\bfone\otimes \bfone\mapsto \bfone$, i.e, $\sLog$ is a \emph{character sheaf}.
\item Consider the $R_1^\times$-torsor of local sections of $\sLog$ that
are modulo $I\sLog$ equal to $1\in \underline{A}$.
Then there is a canonical isomorphism of this $R_1^\times$-torsor  
with $\Log^\times$ such that $\bfone\mapsto 1$. Under this isomorphism
the group structure on $\Log^\times$ is given by the product induced
by the isomorphism in (4). 
\end{enumerate}
\end{proposition}
\begin{proof} (1) follows immediately from Lemma \ref{lemma:gr-hom}
and the functoriality of the functor $M\mapsto \widetilde{M}$.
For (2) note that $ \varphi^*\sLog_{T'}$ are the sections
of $V\times^LR'$, where $L$ acts via $\varphi:L\to L'$ and
$\delta':L'\to (R')^\times$ on $R'$. Then (3) follows from
(2) and \ref{prop:gr-hom}. The assertion (4) follows from the
isomorphism $0^*(\pr_1^*\sLog\otimes_{\underline{R}}pr_2^*\sLog)\isom
0^*(+^*\sLog)$. Finally, as $\sLog=\widetilde{R}$, 
the torsor in (5) is  $\widetilde{R_1^\times} $ and there is
a unique isomorphism with $\Log^\times$ sending $\bfone$ to $1$.
From the remark after Definition \ref{def:torsor} it follows that
the group structure on $\Log^\times$ is induced by the isomorphism
in (4).
\end{proof}

%
\subsection{Trivializations of the logarithm sheaf}
%
\begin{definition}\label{def:mult-triv}
Let $H\subset T$ be a subgroup. A \emph{multiplicative trivialization } of $\sLog$ on $H$ is
a collection of generators $1_h\in \sLog_h$ for all $h\in H$ such that $1_h\bmod I\sLog_h$ equals $1\in A$ and $1_h\otimes 1_{h'}=1_{h+h'}$
under the isomorphism in Proposition \ref{prop:log-properties}
for all $h,h'\in H$.
\end{definition}
We give two alternative descriptions of a
multiplicative trivialization.
First consider the group extension
\[
0\to R^\times_1\to \Log^\times\xrightarrow{\pr_1}T\to 0
\]
from  Definition \ref{def:torsor}. 
A multiplicative trivialization is a group homomorphism $\varrho:H\to \Log^\times$
which is a section of $\pr_1$. In particular, the set of all multiplicative trivializations of $\sLog$
is a $\Hom(H,R^\times_1)$-torsor.

For a second description consider the right translation action $+:T\times H\to T$. A multiplicative trivialization is 
an extension of this $H$-action to $\Log^\times$, i.e.,
a map $\Log^\times\times H\to \Log^\times$
satisfying the usual condition for an $H$-action, 
such that one has a commutative diagram
\[
\begin{CD} \Log^\times \times H@>>> \Log^\times\\
@V\pr_1\times \id VV@VV\pr_1  V\\
T\times H@>+>> T.
\end{CD}
\] 
Given a multiplicative trivialization $\varrho:H\to \Log^\times$
the map $+:\Log^\times\times H\to \Log^\times$ is the composition of $\varrho$ with the group structure $\Log^\times \times \Log^\times\to \Log^\times$. 
\begin{definition}
Denote by $T^\tors:=L_\QQ/L\subset T$ the subgroup of torsion elements in $T$ and
by $T^{(A)}\subset T^\tors$ the subgroup of elements whose order is invertible in $A$.
%
\end{definition}

\begin{proposition}\label{prop:rho-can}
There exists a unique multiplicative trivialization $\varrho_\can$ of $\sLog$ over $T^{(A)}$. It is compatible with
isogenies and for $t\in T[N]\subset T^{(A)}$ it is explicitly given by
the isomorphism
\[
t^*\sLog\isom t^*[N]^*\sLog\isom 0^*[N]^*\sLog\isom 0^*\sLog,
\]
where the outer isomorphisms are the pull-backs of
\ref{prop:log-properties} (3) and the middle one comes
from $[N]\circ t=[N]\circ 0$. 
\end{proposition}

\begin{proof}
Uniqueness: Let $N$ be an integer which is invertible in $A$. 
It suffices to show that $\varrho_\can$ is uniquely determined
on the $N$-torsion points $T[N]$. But the multiplicative
trivializations on $T[N]$ form an $\Hom(T[N],R^\times_1)$-torsor.
But $R^\times_1$ has a filtration
by $(1+I^r)^\times$ such that $\gr^{>0} R_1^\times\isom \Sym^{>0}L_A$,
which has no $N$-torsion as $N$ is invertible in $A$.
This implies
that  $\Hom(T[N],R^\times)=0$.

Existence: Let $\varrho\mid_{T[N]}$ be the inverse of $\Log^\times[N]\isom T[N]$. By construction
these isomorphisms are compatible for different $N$.  
\end{proof}
%
\subsection{Cohomology of the logarithm sheaf}
%
 All unlabelled tensor products in this section and the following
 ones are taken over $\ZZ$.
 
Let $L$ be a lattice of rank $n$. Recall that one has a canonical isomorphism of algebras
$H_\cdot(L,\ZZ)\isom \Lambda^\cdot L$. We define
\begin{equation}
\lambda:=\lambda(L):={\Lambda}^n L=H_n(L,\ZZ).
\end{equation}

\begin{theorem}\label{thm:coh} Let $L$ be lattice of rank $n$.
One has 
\[
H^i(T,\sLog)\isom 
\begin{cases}
0&\mbox{ for }i\neq n\\
H^n(T,\underline{A})&\mbox{ for }i=n
\end{cases}
\]
induced by the map $\sLog\to \sLog/I\sLog = \underline{A}$.
In particular, the cap-product induces an isomorphism
\[
H^n(T,\sLog\otimes \lambda)\isom A.
\] 
\end{theorem}
\begin{proof}
From $\sLog=\underline{R}\otimes_{\underline{A[L]}}\pi_!\underline{A}$ and because $R$ is $A[L]$-flat one gets
$H^i(T,\sLog)\isom H^i(T,\pi_!\underline{A})\otimes_{\underline{A[L]}}\underline{R}$. As
\[
H^i(T,\pi_!\underline{A})\isom H^i_c(V,\underline{A})\isom\begin{cases}
0& i\neq n\\
A& i=n
\end{cases}
\]
this implies the vanishing result. 
The homomorphism $H^n(T,\sLog)\to H^n(T,\underline{A})$
induced by $\sLog\to \underline{A}$ is surjective and because both groups are isomorphic to $A$ it
must be an isomorphism. The cap-product gives
$H^n(T,\sLog\otimes  \lambda)\isom H^n(T,\underline{A})\otimes H_n(L,\ZZ)\isom A$.  
\end{proof}

\begin{corollary}\label{cor:coh-open} Let $D\subset T$ be a finite and non-empty subset. 
Then for $i\neq n-1$
\[
H^i(T\setminus D, \sLog\otimes  \lambda)=0
\] 
and
one has a short exact sequence 
\[
0\to H^{n-1}(T\setminus D, \sLog\otimes   \lambda)\xrightarrow{\res}\sLog\mid_D\xrightarrow{\sigma_D}A\to 0,
\]
where $\sLog\mid_D=\bigoplus_{d\in D}\sLog_d$ is the restriction of $\sLog$ to $D$ and $\sigma_D$
is the sum of the maps $\sLog_d\to \sLog_d/I\sLog_d= A$. 
\end{corollary} 
\begin{proof}
Consider the localization sequence for the closed subset $D\subset T$
\[
\cdots\to H^i (T,\sLog\otimes   \lambda)\to H^i(T\setminus D,\sLog\otimes   \lambda)
\to H^{i+1}_D(T,\sLog\otimes   \lambda)\to \cdots
\]
For each $d\in D$ choose an open neighbourhood $U_d$ such that 
$ \sLog$ is constant on $U_d$ and the $U_d$ for 
different $d$ are disjoint. Then by excision 
\[
H^{i+1}_D(T,\sLog\otimes   \lambda)\isom
\bigoplus_{d\in D}H^{i+1}_{\{d\}}(U_d,\sLog\mid_{U_d}\otimes \lambda).
\]
As $\sLog\mid_{U_d}$ is constant and hence isomorphic to
$\sLog_{d}$, one has a canonical isomorphism
\[
H^{n}_{\{d\}}(U_d,\sLog\mid_{U_d}\otimes \lambda)\isom
\sLog_d
\]
and
$H^{i+1}_{\{d\}}(U_d,\sLog\mid_{U_d}\otimes \lambda)=0$ 
for $i+1\neq n$ (see \cite[Proposition 3.2.3]{KashiwaraSchapira}).
\end{proof}

%
\subsection{Equivariant cohomology of the logarithm sheaf}
%
We describe an equivariant version of the above construction.

Let $\Gamma\to \GL(L)$ be a group action on $L$.
We write $L\rtimes \Gamma$ for the semi-direct product
with multiplication
\[
(\ell,\gamma)(\ell',\gamma')=(\ell+\gamma\ell',\gamma\gamma').
\]
To follow our principle, we let $(l,\gamma)\in L\rtimes \Gamma$ act from the right on $v\in V$ by
$v(\ell,\gamma):=\gamma^{-1}v+\gamma^{-1}\ell$. In particular, the $L$-torsor $\pi:V\to T$ is
$\Gamma$-equivariant. 
From this we deduce a right action of $\Gamma$ on $\Log$ by
\begin{align}
\Log\times \Gamma\to \Log;&&((v,r),\gamma)\mapsto(\gamma^{-1}v,\varphi_{\gamma^{-1}}(r))
\end{align}
so that $\sLog$ is a $\Gamma$-equivariant sheaf. We want to compute the $\Gamma$-equivariant cohomology
$H^i(T,\Gamma;\sLog\otimes  \lambda)$ but for later needs, we
compute a slightly more general cohomology group. 
\begin{theorem}\label{thm:coh-equiv}
Let $D\subset T$ be a finite non-empty subset stabilized by $\Gamma$
and $M$ an $A[\Gamma]$-module. Then:
\begin{enumerate}
\item There are isomorphisms
\[
H^i(T,\Gamma;\underline{M}\otimes_{\underline{A}}\sLog\otimes  \lambda)\isom H^{i-n}(\Gamma,M)
\]
and
\[
H^i(T\setminus D,\Gamma;\underline{M}\otimes_{\underline{A}}\sLog\otimes  \lambda)\isom H^{i-n+1}(\Gamma, H^{n-1}(T\setminus D,\underline{M}\otimes_{\underline{A}}\sLog\otimes  \lambda)).
\]
\item One has a long exact sequence
\[
...\to H^i(T\setminus D,\Gamma;\underline{M}\otimes_{\underline{A}}\sLog\otimes  \lambda)\xrightarrow{\res}H^{i-n+1}(\Gamma,{M}\otimes_{{A}}\sLog\mid_D)
\xrightarrow{\sigma_D}H^{i-n+1}(\Gamma,M)\to ...
\]
\end{enumerate}
\end{theorem}
\begin{proof}
This is follows from the spectral sequence 
\[
H^i(\Gamma,H^j(X, \underline{M}\otimes_{\underline{A}}\sLog\otimes  \lambda))\Rightarrow
H^{i+j}(X,\Gamma;\underline{M}\otimes_{\underline{A}}\sLog\otimes  \lambda)
\]
for $X=T,T\setminus D$, the isomorphism 
$ H^j(X, \underline{M}\otimes_{\underline{A}}\sLog\otimes  \lambda)\isom M\otimes_A H^j(X, \sLog\otimes  \lambda)$
of $A[\Gamma]$-modules, Theorem \ref{thm:coh} and Corollary \ref{cor:coh-open}.
\end{proof}
As a special case, we get:
\begin{corollary}\label{cor:equiv-coh}
One has $H^i(T\setminus D,\Gamma;\underline{M}\otimes_{\underline{A}}\sLog\otimes  \lambda)=0$ for $i<n-1$ and a canonical isomorphism
\[
\res:H^{n-1}(T\setminus D,\Gamma;\underline{M}\otimes_{\underline{A}}\sLog\otimes  \lambda)\isom \ker(M\otimes_A\sLog\mid_D\xrightarrow{\sigma_d}M)^{\Gamma}.
\]
\end{corollary}
%
\subsection{The topological polylogarithm and Eisenstein classes}
%
\begin{definition}
For $D\subset T$ finite and non-empty we define
\[
A[D]^0:=\ker(\bigoplus_{d\in D}A\xrightarrow{\sum}A),
\]
where $\Sigma$  is the
summation map $(a_d)_{d\in D}\mapsto\Sigma_{d\in D}a_d$. 
We view the elements $\alpha\in A[D]^0$ as functions $\alpha:D\to A$.
We also set 
\[
R[D]^0:=\ker(\bigoplus_{d\in D}R\xrightarrow{\sigma_D}A)
\]
where $\sigma_D$ is the sum of the augmentations $R\to R/IR=A$.
\end{definition}
Suppose that $D\subset T^{(A)}$ and that $\Gamma$ stabilizes
$D$. Then the trivialization $\varrho_\can$ from Proposition \ref{prop:rho-can} induces an isomorphism $R[D]^0\isom\ker(\sLog\mid_D\xrightarrow{\sigma_D}A)$, so that
we get 
\[
(A[D]^0)^\Gamma\subset (R[D]^0)^\Gamma\isom\ker(\sLog\mid_D\xrightarrow{\sigma_D}A)^\Gamma.
\] 
We apply this to Corollary \ref{cor:equiv-coh} in the case $M=A$:
\begin{definition}\label{def:pol-def}
For $D\subset T^{(A)}$, stabilized by $\Gamma$ and $\alpha\in (A[D]^0)^\Gamma$ the unique cohomology class 
\[
\pol_\alpha\in   H^{n-1}(T\setminus D,\Gamma;\sLog\otimes   \lambda)
\]
with $\res(\pol_\alpha)=\alpha$ is called the \emph{topological polylogarithm}
associated with $\alpha$.
\end{definition}
\begin{remark} Note that $(A[D]^0)^\Gamma\neq 0$ in general: 
Let $N$ be invertible in $A$ and $D=T[N]$ be the $N$-torsion points of $T$. Then $D$ is stable under
$\Gamma$ and $N^n \delta_0-\sum_{d\in T[N]}\delta_d\in (A[D]^0)^\Gamma$. 
\end{remark}
Let $t\in T\setminus D$ be any point stabilized
by $\Gamma$. Then the pull-back of $\pol_\alpha$ along $t$ is a
cohomology class
\begin{equation}
t^*\pol_\alpha\in H^{n-1}(\Gamma, \sLog_t\otimes \lambda).
\end{equation}
If $t\in T^{(A)}$, we can use  the 
trivialization $\varrho_\can$ to identify $\sLog_t\isom R$.
\begin{definition}\label{def:Eisenstein-class}
Let $D\subset T^{(A)}$ and $t\in T^{(A)}\setminus D$ be both stabilized by 
$\Gamma$, then for $\alpha\in (A[D]^0)^\Gamma$ the class 
\begin{equation*}
\Eis_\alpha(t):=t^*\pol_\alpha\in H^{n-1}(\Gamma, R\otimes \lambda).
\end{equation*}
is called the \emph{Eisenstein classes} associated to $t$
and $\alpha$.
\end{definition}
If we identify $\widehat{\Gamma}L_A\isom \widehat{\TSym}L_A$
then we also write
\begin{equation}
\Eis^k_\alpha(t):=\exp^*_k(\Eis_\alpha(t))\in H^{n-1}(\Gamma,\TSym^kL_A\otimes \lambda)
\end{equation}
for the $k$-th component of $\exp^*(\Eis_\alpha(t))$. 

The following special case of the above definition  was considered by Nori and Sczech: 
\begin{definition}\label{def:Eisenstein-special-case}
Let $D\subset T^{(A)}$ be a finite non-empty subset such that $0\notin D$. The Eisenstein operator of Nori and Sczech is the map
\begin{align*}
(A[D]^0)^\Gamma&\to H^{n-1}(\Gamma,R\otimes\lambda)\\
\alpha&\mapsto \Eis_\alpha(0).
\end{align*}
\end{definition}
As above, the $k$-th component of $\exp^*(\Eis_\alpha(0))$ gives 
rise to a map
\begin{align*}
(A[D]^0)^\Gamma\to H^{n-1}(\Gamma,\TSym^k L_A\otimes\lambda)
&&\alpha\mapsto \Eis^k_\alpha(0).
\end{align*}

%
\subsection{A variant of the polylogarithm I}
%

For the study of the general Eisenstein distribution later
the polylogarithm defined so far is not flexible enough.
In this section we discuss the required slight generalization of the polylogarithm.

Let $C\subset T$ be a finite subset then $\sLog\mid_C$ has an $A[C]$-module structure
\begin{equation}\label{eq:AC-module-str}
A[C]\otimes \sLog\mid_C\to \sLog\mid_C
\end{equation}
given on a stalk $c\in C$ by multiplication with the value $f(c)$
for $f\in A[C]$. Assume that $C\subset T^{\tors}$, 
$C\cap D=\emptyset$ and
suppose that $\Gamma$ stabilizes $C$ and $D$. Let $M=A[C]$, then
from Corollary \ref{cor:equiv-coh} we get the isomorphism
\begin{equation}\label{eq:pol-isom-var-I}
\res:H^{n-1}(T\setminus D,\Gamma;\underline{A}[C]\otimes_{\underline{A}}\sLog\otimes  \lambda)\isom \ker(A[C]\otimes_A\sLog\mid_D\xrightarrow{\sigma_d}A[C])^{\Gamma}.
\end{equation}
From the  definition of $A[D]^0$ we get 
\[
(A[C]\otimes_A A[D]^0)^\Gamma\subset 
\ker(A[C]\otimes_A\sLog\mid_D\xrightarrow{\sigma_d}A[C])^{\Gamma}.
\]
\begin{definition}\label{def:pol-def-var-I}
We define for $h\in (A[C]\otimes_A A[D]^0)^\Gamma$
the polylogarithm $\pol_h$ to be the class
\[
\pol_h\in H^{n-1}(T\setminus D,\Gamma;\underline{A}[C]\otimes_{\underline{A}}\sLog\otimes  \lambda)
\]
which corresponds to $h$ under the isomorphism
\eqref{eq:pol-isom-var-I}.
\end{definition}
The restriction of $\pol_h$ to $C$ is 
a class in $H^{n-1}(\Gamma, A[C]\otimes_A\sLog\mid_C\otimes\lambda)$
and the image under the map 
from \eqref{eq:AC-module-str} gives a class 
\begin{equation}\label{eq:Eis-f-alpha}
\Eis_h\in H^{n-1}(\Gamma,\sLog\mid_C\otimes \lambda).
\end{equation}
\begin{definition}\label{def:Eis-as-map}
For $C\subset T^\tors$ and $D\subset T^{(A)}$ with $C\cap D=\emptyset$ and such that $\Gamma$ stabilizes $C$ and  $D$, we define the map
\[
\Eis:(A[C]\otimes_A A[D]^0)^{\Gamma}\to H^{n-1}(\Gamma,\sLog\mid_C\otimes \lambda)
\]
by $h\mapsto\Eis_h$.
\end{definition}
\begin{remark}
A more intuitive way to think about $\Eis_h$ is as follows. Suppose
that $\Gamma$ stabilizes each point of $C$. Then we can
view $h\in (A[C]\otimes_A A[D]^0)^{\Gamma}$ as a map
$h:C\to (A[D]^0)^\Gamma$, $c\mapsto h_c$ with $h_c(d):=h(c,d)$.
With this notation one has
$\Eis_h=\sum_{c\in C}c^*\pol_{h_c}$, with $\pol_{h_c}$ 
as defined in \ref{def:pol-def}.
\end{remark}
%
\subsection{A variant of the polylogarithm II}
%
The polylogarithm $\pol_\alpha$ has the advantage of being defined
for arbitrary coefficients and it has good trace compatibilities 
as will be shown in the next section. The disadvantage is that it
depends on functions $\alpha$ of degree zero.
 The variant 
$\pol$ discussed below can be evaluated on each non-zero torsion point
but works only for $\QQ$-algebras $A$. It is also this version
of the polylogarithm which plays the dominant role in the literature on the motivic polylogarithm.

We specialize Corollary \ref{cor:equiv-coh} to the case
$D:=\{0\}$ and $M=L_A^*:=\Hom_A(L_A,A)$. Then we get
\begin{equation}\label{eq:pol-isom-var-II}
\res:H^{n-1}(T\setminus \{0\},\Gamma;\underline{L}_A^*\otimes_{\underline{A}}\sLog\otimes  \lambda)\isom (L_A^*\otimes_AI)^{\Gamma},
\end{equation}
where $I\subset R$ is the augmentation ideal. 
If $A$ is a $\QQ$-algebra we have an isomorphism $\exp^*:R\isom \widehat{\Sym}L_A$ and  we have
a canonical class
\begin{equation}\label{eq:varpi-def}
\varpi\in L_A^*\otimes_A L_A\subset L_A^*\otimes_A I
\end{equation}
corresponding to $\id:L_A\to L_A$. Obviously, 
$\varpi\in (L_A^*\otimes I)^\Gamma$.
\begin{definition}\label{def:pol-variant}
Let $A$ be a $\QQ$-algebra, then the polylogarithm $\pol$ is the
class
\[
\pol\in H^{n-1}(T\setminus\{0\},\Gamma;\underline{L}_A^*\otimes_A\sLog\otimes\lambda)
\]
corresponding to $\varpi$ under the isomorphism 
\eqref{eq:pol-isom-var-II}.
\end{definition}
The contraction $L_A^*\otimes_A \Sym^kL_A\to \Sym^{k-1}L_A$
induces a map 
\begin{equation}\label{eq:contr-def}
\contr:L_A^*\otimes_AR\to R.
\end{equation}
Furthermore, the multiplication $L_A\otimes_A R\to R$ 
induces 
\begin{equation}\label{eq:mult-def}
\mult: R\to L_A^*\otimes_A R
\end{equation}
and it is straightforward to show that $\contr\circ\mult=\id$.
The map $\mult$ extends to a homomorphism of sheaves
\begin{equation}
\mult:\sLog\to 
\underline{L}_A^*\otimes_A\sLog
\end{equation}
Let $t\in T^\tors\setminus\{0\}$ be
stabilized by $\Gamma$. Then $\varrho_\can$ allows us to 
identify  $t^*\sLog\isom R$.
\begin{definition}\label{def:Eisenstein-variant}
Let $t\in T^\tors\setminus\{0\}$ be
stabilized by $\Gamma$. 
The class
\[
\Eis(t):=\contr(t^*\pol)\in H^{n-1}(\Gamma,R\otimes \lambda)
\]
is called
the \emph{Eisenstein class} associated to $t$. We also write
\[
\Eis^k(t):=\exp^*_k(\Eis(t))\in  H^{n-1}(\Gamma,\Sym^kL_A \otimes\lambda).
\]
\end{definition}
Let us discuss one special case of the relation between 
$\Eis^k(t)$ and the class $\Eis_\alpha^k(t)$ defined in
\ref{def:Eisenstein-class}, which will be used later
(compare also \cite[12.4.4]{Kings-Eisenstein}).
\begin{definition}\label{def:alpha-varphi-def}
Let $\varphi:L\to L'$ be an isogeny and define the 
function on $D:=L'/\varphi(L)$
\begin{equation*}
\alpha_\varphi:=(\deg\varphi)\delta_0-\sum_{d\in D}\delta_d.
\end{equation*}
\end{definition}
Consider 
\[
\mult(\pol_{\alpha_\varphi})\in
H^{n-1}(T\setminus\varphi^{-1}(0),\Gamma;\underline{L}_A^*\otimes_A\sLog\otimes\lambda)
\]
then using the isomorphisms $\sLog\isom\varphi^*\sLog'$ and
$L_A\isom L'_A$ (because $A$ is a $\QQ$-algebra)
one also has ($D=\varphi^{-1}(0)$)
\[
\varphi^*\pol'\in H^{n-1}(T\setminus D,\Gamma;{\underline{L}}_A^*\otimes_A\sLog\otimes\lambda).
\]
Finally, $\pol\mid_{T\setminus D}$, the restriction of $\pol$ to $T\setminus D$,
gives a class in the same group.
\begin{proposition}
One has the equality
\[
\mult(\pol_{\alpha_\varphi})=
(\deg\varphi) \pol\mid_{T\setminus D}-\varphi^*\pol'.
\]
in $H^{n-1}(T\setminus D,\Gamma;{\underline{L}}_A^*\otimes_A\sLog\otimes\lambda)$.
\end{proposition}
\begin{proof}
From  Theorem \ref{thm:coh-equiv} we have an isomorphism
\[
\res:H^{n-1}(T\setminus D,\Gamma;{\underline{L}}_A^*\otimes_{\underline{A}}\sLog\otimes\lambda)\isom
(L_A^*\otimes_AR[D]^0)^\Gamma.
\]
We have
$\res(\mult(\pol_{\alpha_\varphi}))=(\deg\varphi)\delta_0\varpi- \sum_{d\in D}\delta_d\varpi$ and
$\res((\deg\varphi) \pol\mid_{T\setminus D})=
(\deg\varphi)\delta_0\varpi$. Moreover,
$\res(\varphi^*\pol')=\sum_{d\in D}\delta_d\varpi$, which proves the claim.
\end{proof}
\begin{corollary}\label{cor:Eisenstein-comparison}
For $k\ge 0$
the relation of Eisenstein classes 
\[
\Eis_{\alpha_\varphi}^k(t)=(\deg\varphi)\Eis^k(t)-{\Eis'}^k(\varphi(t))
\]
holds in $H^{n-1}(\Gamma,\TSym^kL_A\otimes\lambda)$, where
we have used the isomorphism $\Sym^kL_A\isom \TSym^kL_A$.
\end{corollary}
\begin{proof}
One has
\begin{align*}
\Eis_{\alpha_\varphi}^k(t)&= \exp^*_k(\contr\circ \mult (t^*\pol_{\alpha_\varphi}))\\
&= (\deg\varphi) \exp^*_k\circ\contr\left(t^*\pol\mid_{T\setminus D}
-t^*\varphi^*\pol'\right)\\
&=(\deg\varphi)\Eis^k(t)-{\Eis'}^k(\varphi(t)).
\end{align*}
\end{proof}

%
\subsection{Trace compatibility}
%

The polylogarithm classes are compatible  with respect to isogenies $\varphi_T:T'\to T$ (note that in this section we 
interchange the role of $L$ and $L'$).
This is a geometric incarnation of the distribution property 
of Eisenstein series. 

We use the following set up:
Let $L,L'$ be lattices of rank $n$ 
with actions by $\Gamma$ and let $\varphi:L'\to L$ be an isogeny compatible with the $\Gamma$-action.
Then one
has a group homomorphism $(\varphi,\id):L'\rtimes \Gamma\to L\rtimes \Gamma$.

We consider
finite non-empty subsets
$D\subset T^{(A)}$ and
$D'\subset {T'}^{(A)}$ such that $\varphi_T(D')\subset D$.
One has a cartesian square
\begin{equation}\label{eq:cart-sq}
\begin{CD}
\varphi^{-1}(D)@>>> T'\\
@V\varphi VV@VV\varphi V\\
D@>>> T
\end{CD}
\end{equation}

\begin{proposition}\label{prop:tr-diagr} Let $M$ be an $A[\Gamma]$-module,
then there is a trace map
\[
\Tr_\varphi:H^{n-1}(T'\setminus D',\Gamma; \underline{M}\otimes_{\underline{A}}\sLog_{T'}\otimes   \lambda')\to
H^{n-1}(T\setminus D,\Gamma; \underline{M}\otimes_{\underline{A}}\sLog_{T}\otimes   \lambda)
\]
such that the diagram 
\[
\begin{CD}
H^{n-1}(T'\setminus D',\Gamma; \underline{M}\otimes_{\underline{A}}\sLog_{T'}\otimes   \lambda')@>\res >\isom> \ker(M\otimes_A R'[D']\to M)^{\Gamma}\\
@V\Tr_\varphi VV @VV\varphi_RV\\
H^{n-1}(T\setminus D,\Gamma; \underline{M}\otimes_{\underline{A}}\sLog_{T}\otimes   \lambda)@>\res >\isom> \ker(M\otimes_AR[D]\to M)^{\Gamma}
\end{CD}
\]
commutes.
\end{proposition}
\begin{proof}
As $\varphi:T'\to T$ is a topological submersion and a finite map we have $\varphi^*(\underline{M}\otimes_{\underline{A}}\sLog_{T})\otimes  \lambda'\isom \varphi^!(\underline{M}\otimes_{\underline{A}}\sLog_{T})\otimes \lambda$ (see \cite[Section 3.3]{KashiwaraSchapira}). 
In particular, the trace map 
$R\varphi_!\varphi^!(\underline{M}\otimes_{\underline{A}}\sLog_{T})\to
\underline{M}\otimes_{\underline{A}}\sLog_{T}$ induces a map
\[
\varphi_!\varphi^*(\underline{M}\otimes_{\underline{A}}\sLog_{T})\otimes \lambda'\to \underline{M}\otimes_{\underline{A}}\sLog_{T}\otimes \lambda.
\]
This gives 
\begin{align*}
H^{n-1}(T'\setminus D',\underline{M}\otimes_{\underline{A}}\sLog_{T'}\otimes  \lambda')&\xrightarrow{\varphi_\sLog} 
H^{n-1}(T'\setminus D',\varphi^*(\underline{M}\otimes_{\underline{A}}\sLog_{T})\otimes  \lambda')\\
&\xrightarrow{\restr} H^{n-1}(T'\setminus \varphi^{-1}(D),\varphi^*(\underline{M}\otimes_{\underline{A}}\sLog_{T})\otimes  \lambda')\\
&\xrightarrow{\isom}  H^{n-1}(T\setminus D,\varphi_!\varphi^*(\underline{M}\otimes_{\underline{A}}\sLog_{T})\otimes  \lambda')\\
&\to H^{n-1}(T\setminus D, \underline{M}\otimes_{\underline{A}}\sLog_{T}\otimes   \lambda),
\end{align*}
where we have used that  $\varphi:T'\setminus \varphi^{-1}(D)\to T\setminus D$ is finite, so that
$\varphi_!=\varphi_*$. The result follows from Theorem \ref{thm:coh-equiv} and the diagram  commutes 
because of the cartesian square \ref{eq:cart-sq} and \cite[3.1.9]{KashiwaraSchapira}.
\end{proof}
\begin{remark}
In this paper we consider only the trace compatibility for isogenies. We remark that
a similar statement holds also in the more general case of a submersion. This was used in \cite{Kings08} to compute the residue of
the Eisenstein classes on Hilbert modular varieties.
\end{remark}
We discuss now the consequences of this proposition for the different
notions of polylogarithm we have defined.
\begin{corollary}\label{cor:trace-compatibility} In the situation
of Definition \ref{def:pol-def} one has for 
$\alpha\in (A[D']^0)^\Gamma$ 
\[
\Tr_\varphi(\pol'_\alpha)=\pol_{\varphi_*(\alpha)}
\]
where $\varphi_*(\alpha)$ is the function
$\varphi_*(\alpha)(d)=\sum_{d'\in \varphi^{-1}(d)}\alpha(d')$.
\end{corollary}
\begin{proof}
This is immediate from the definition, Proposition
\ref{prop:tr-diagr} and the fact that the
restriction of $\varphi$ to the subspace
$(A[D']^0)^\Gamma\subset (R[D']^0)^\Gamma$ is given by
the formula in the corollary.
\end{proof}

The following generalization of the trace compatibility is used
later in the general study of Eisenstein distributions.
In the situation of Proposition \ref{eq:cart-sq}
assume in addition that one has finite non-empty subsets $C\subset T^\tors$, $C'\subset {T'}^\tors$ 
with $\varphi_T(C')\subset C$ and $C\cap D=\emptyset$, $C'\cap D'=\emptyset$. We have 
\begin{equation}
\begin{CD}
C'\amalg D'@>>>\varphi^{-1}(C)\amalg\varphi^{-1}(D)@>>> T'\\
 @.@V\varphi VV@VV\varphi V\\
 @. C\amalg D@>>> T
\end{CD}
\end{equation}
We assume that 
$\Gamma$ stabilizes $C\cup D$ 
and $C'\cup D'$. Then the trace map
$\Tr_\varphi$ induces a homomorphism $A[C']\to A[C]$, which we call
$\varphi$ (it is the same as $\varphi:A[D']\to A[D]$). Let
\[
\varphi\circ\Tr_\varphi:H^{n-1}(T'\setminus D',\Gamma; \underline{A}[C']\otimes_{\underline{A}}\sLog_{T'}\otimes   \lambda')\to
H^{n-1}(T\setminus D,\Gamma; \underline{A}[C]\otimes_{\underline{A}}\sLog_{T}\otimes   \lambda)
\]
be the composition of the trace map $\Tr_\varphi$ with
the map induced by $\varphi$.
Recall the Eisenstein operator 
\[
\Eis:(A[C]\otimes_A A[D]^0)^\Gamma\to
H^{n-1}(\Gamma,\sLog\mid_{C}\otimes \lambda)
\]
from Definition \ref{def:Eis-as-map}. The trace compatibility for
$\pol_h$ has the following consequence for $\Eis$.
\begin{corollary}\label{cor:Eis-trace-comp}
Let $C'=\varphi^{-1}(C)$ and $\varphi^*:A[C]\to A[C']$ be the map 
$f\mapsto f\circ \varphi$. Then for $h\in (A[C]\otimes_AA[D']^0)^\Gamma$
one has 
\[
\Tr_\varphi(\Eis'_{(\varphi^*\otimes \id)(h)})=\Eis_{(\id\otimes \varphi_*)(h)}.
\]
\end{corollary}
\begin{proof}
This is immediate from the definition of $\Eis$ and the commutative diagram 
\[
\xymatrix{
A[C]\otimes \sLog'\mid_{C'}\ar[r]^{\varphi^*\otimes \id}\ar[d]_{\id\otimes \Tr_\varphi}& A[C']\otimes\sLog'\mid_{C'}\ar[r]^/.8em/{\eqref{eq:AC-module-str}}&
\sLog'\mid_{C'}\ar[d]^{\Tr_\varphi}\\
A[C]\otimes \sLog\mid_{C}\ar[rr]^{\;\eqref{eq:AC-module-str}}&&\sLog\mid_C.
}
\]
\end{proof}

%
\section{Explicit formulas}
%
In this section we give an explicit formula for the
topological polylogarithm. The computations 
were essentially done by Nori \cite{Nori} and we present
them here in a slightly different form.

In this section we always consider $A=\CC$ so that we can 
identify $R\isom \widehat{\Sym}^\cdot L_\CC$.
%
\subsection{The continuous trivialization of the logarithm sheaf}
%

Let $\sP$ be the space of positive definite symmetric bilinear forms on $V$, which we also
consider as translation invariant metrics on $V$. Let $V^*$ be the $\RR$-dual of $V$, then
we consider $\sP\subset \Hom(V,V^*)$ with its induced right $\Gamma$-action:
$B[\gamma](v,w):=B(\gamma v,\gamma w)$.

We let $ L\rtimes\Gamma$ act on $\sP\times V$ by
\[
(B,v)(\ell,\gamma):=( B[\gamma],\gamma^{-1}(v+\ell)).
\]
Note that the action of $\Gamma$ on $L$ factors
through $\GL(L)$, which acts almost discretely on $\sP$ and that $\sP$
is contractible.
We consider the sheaf $\sLog$ over $(\sP\times\Log)/\Gamma$
as the sheaf of sections of 
$\sP\times V/L\rtimes\Gamma$. Let $D\subset T^{(A)}$ be
a finite non-empty subset, which does not contain $0$. 
By general principles from
equivariant cohomology we have 
\[
H^{n-1}(T\setminus D, \Gamma, \sLog\otimes \lambda)\isom
H^{n-1}(\sP\times (V\setminus \pi^{-1}(D))/L\rtimes\Gamma,\sLog\otimes \lambda).
\]
We need to set up some more notation.

We write 
$\sLog^\infty$ and $\underline{R}^\infty$ for
the $\sC^\infty$-pro-bundles defined by the inverse system of bundles
that correspond to the local systems $\sLog^{(k)}$ and $\underline{R}^{(k)}$. We consider the sheaves of $\sC^\infty$-differential forms $\Omega^i \widehat{\otimes}\sLog:=\prolim_k\Omega^i\otimes\sLog^{(k)}$
and currents $\widehat{\Omega}^i\widehat{\otimes}\sLog:=\prolim_k \widehat{\Omega}^i\otimes \sLog^{(k)}$ etc. with values in these pro-bundles. Recall that a current is a generalized section of the 
sheaf of $k$-forms. 


On $\sLog^\infty$ we have the connection $\nabla:=d\otimes \id$ and on $\underline{R}^\infty$ we have the connection
$\nabla_0:=d\otimes \id$. 
\begin{definition}
Let $\kappa$ be the $\underline{R}^\infty$-valued $1$-form on $T$  
\[
\kappa\in V^*\otimes V\subset V^*\widehat{\otimes}R\subset\Gamma(T,\Omega^1_T\widehat{\otimes}\underline{R}),
\]
which corresponds to the identity map $\id\in \Hom(V,V)\isom 
V^*\otimes V$.
\end{definition}
Obviously, $\kappa$ is $\Gamma$ invariant.
\begin{lemma}\label{lemma:cont-triv}
The sheaf $\sLog^\infty$ admits a unique continuous multiplicative trivialization $\varrho_\cont$
on $T$. The
section $\varrho_\cont$ is $\sC^\infty$, compatible with $N$-multiplication, and one has 
\[
\nabla(\varrho_\cont)=-\kappa\varrho_\cont.
\]
In particular, $\varrho_\cont\mid_{T^\tors}=\varrho_\can$.
\end{lemma}
\begin{proof}
The set of continuous multiplicative trivializations with the property in the lemma is
a torsor under $\Hom_\cont(T, (1+I)^\times)$, which is trivial because $T$ is compact.
This shows the uniqueness of $\varrho_\cont$. For the existence we consider
\begin{align*}
V\to V\times R^\times&& v\mapsto(v, \exp(-v))
\end{align*}
where $\exp(-v):=\sum_{k\ge 0}\frac{(-v)^{\otimes k}}{k!}$. This is a section 
by \eqref{eq:ass-sheaf}, is compatible with $N$-multiplication and has the desired
property $d\exp(-v)=-\kappa\exp(-v)$.
\end{proof}
%
\subsection{Green's currents and the topological polylogarithm}
%

We use $\varrho_\cont$ from Lemma \ref{lemma:cont-triv}
to identify $\varrho_\cont:\underline{R}^\infty\isom \sLog^\infty$.
The connection $\nabla$ of $\sLog^\infty$ corresponds to $\nabla_0-\kappa$ under this identification.
In particular, we can compute the equivariant cohomology of $\sLog$ as 
\[
H^i(T\setminus D,\Gamma;\sLog\otimes\lambda)=H^i(({\Omega}^\cdot(\sP\times (T\setminus D))\widehat{\otimes}R\otimes \lambda)^\Gamma,\nabla).
\]
For the construction of a cohomology class representing the topological polylogarithm $\pol_\alpha$, 
we will first construct a certain Green's-current. To define
these, we need two notations: Let $\lambda^*:=\Hom_\ZZ(\lambda,\ZZ)$, then 
the volume form on $T$ is defined to be the section
\begin{equation}
\vol\in \lambda^*\otimes\lambda\subset \Omega^n(T)\otimes \lambda
\end{equation}
corresponding to the isomorphism $\lambda\isom \lambda$.
Let $\delta_{\sP\times\{0\}}$ be the 
delta function of $\sP\times \{0\}\subset \sP\times T$.
We consider this as an element in $\widehat{\Omega}^n(T)\otimes\lambda$ by multiplying it with
$\vol$.
\begin{definition}\label{def:Green-current}
A \emph{Green's-current} is an $n-1$-current 
$\sG\in  (\widehat{\Omega}^{n-1}(\sP\times T)\widehat{\otimes}R\otimes \lambda)^\Gamma$, which is
smooth on $\sP\times (T\setminus\{0\})$, and such that
\[
\nabla(\sG)=\delta_{\sP\times\{0\}}\vol-\vol
\]
in $(\widehat{\Omega}^{n}(\sP\times T)\widehat{\otimes}R\otimes \lambda)^\Gamma$.
\end{definition} 
With a method due essentially to Nori we prove in the next 
section (see Corollary \ref{cor:Epstein-zeta}):
\begin{theorem}\label{thm:existence-green-currents}
A Green's-current as in Definition \ref{def:Green-current} exists. 
\end{theorem}
 Here we explain how we get a representative of $\pol_\alpha$ with the help of $\sG$.
The group $T$ acts on the complex $\widehat{\Omega}^{\cdot}(\sP\times T)\widehat{\otimes}R$
by translation. 
\begin{definition} Let $D\subset T^\tors$ be finite and non-empty and 
$\sG$ be a Green's-current. Let $\tau_d$ be the translation by $d\in T$ and $\alpha=\sum_{d\in D}\alpha_d1_d\in \CC[D]^0$.
Then we define the $n-1$-current
\[
\sG(\alpha):=\sum_{d\in D}\alpha_d\tau_{-d}^*\sG, 
\]
which is smooth on $\sP\times (T\setminus D)$.
\end{definition}
With this notation we can formulate the main result in
this section.
\begin{theorem}\label{thm:green-currents}
If $D\subset T^\tors$ and $\alpha\in (\CC[D]^0)^\Gamma$, then the restriction of
$\sG(\alpha)$ to $\sP\times (T\setminus D)$ is a smooth $\Gamma$-invariant closed $n-1$-form,
which represents $\pol_\alpha$. 
\end{theorem}
\begin{proof}
As $\alpha$ and $\sG$ are $\Gamma$-invariant, the same holds for $\sG(\alpha)$. By definition
$\nabla(\sG(\alpha))=\sum_{d\in D}\alpha_d\delta_d$, which implies that the restriction of
$\sG(\alpha)$ to $\sP\times (T\setminus D)$ is closed and that $\res(\sG(\alpha))=\alpha$.
With Corollary \ref{cor:equiv-coh} we see that $\sG(\alpha)$ represents $\pol_\alpha$.
\end{proof}
We also want to construct a current, which represents the
variant of the polylogarithm $\pol$ from Definition \ref{def:pol-variant}. Let $\ell_1,\ldots,\ell_n$ be a basis
of $L$ and $\mu_1,\ldots,\mu_n$ be the dual basis of $V^*$. 
Define the closed form 
$\eta:=\frac{1}{n}\sum_{j=1}^n(-1)^j\mu_j d\mu_1\wedge\ldots
\wedge\widehat{d\mu_j}\wedge\ldots\wedge d\mu_n$, then
a straightforward computation shows
\[
-\kappa\eta=\varpi\vol,
\]
where $\varpi\in L_A^*\otimes_AL_A$ is element from \eqref{eq:varpi-def}.
\begin{theorem}\label{thm:green-current-variant}
Let $\widetilde{\sG}:=\varpi\sG+\eta$, then
\[
\nabla(\widetilde{\sG})=\delta_{\sP\times\{0\}}\varpi\vol
\]
and $\widetilde{\sG}$ represents $\pol\in 
H^{n-1}(T\setminus\{0\},\Gamma;\underline{L}_A^*\otimes_A\sLog\otimes\lambda)$ defined in \ref{def:pol-variant}.
\end{theorem}
\begin{proof}
This follows from the formula 
\[
(\nabla_0-\kappa)(\varpi\sG+\eta)=\varpi(\delta_{\sP\times\{0\}}\vol-\vol)-\kappa\eta=\varpi\delta_{\sP\times\{0\}}\vol.
\]
\end{proof}
%
\subsection{Explicit construction of a Green's current}
%
The idea for the construction of the Green's current presented in this section goes essentially back to Nori  \cite{Nori}. 
One rewrites $\sG$ as a Fourier-series and considers the resulting differential
equations for the coefficients. This differential equation can be solved by inverting a differential operator.

We write 
\[
L^*:=\Hom(L,\ZZ)\subset V^*
\]
for the dual lattice of $L$ and $\langle,\rangle:V\times V^*\to \RR$ for the evaluation map.
Further we let $i:=\sqrt{-1}$ be a square-root of $-1$.
Any current $\sG\in \widehat{\Omega}^{n-1}(\sP\times T)\widehat{\otimes}R\otimes \lambda$ has a Fourier-series
\[
\sG(B,v)=\sum_{\mu\in L^*}E_\mu(B)e^{2\pi i\langle v,\mu\rangle}
\]
where $E_\mu(B)\in \Omega^{n-1}(\sP\times T)\widehat{\otimes}R\otimes \lambda$ are $R$-valued differential forms, which are constant in the $T$ direction. We write 
\[
E_\mu(B)=E_\mu^0+\ldots+E_\mu^{n-1}
\]
and 
$E_\mu^a\in \Omega^a(\sP)\otimes \Lambda^{n-1-a}V^*\widehat{\otimes}R\otimes \lambda$ is
the component in bidegree $(a,n-1-a)$ of $E_\mu$.
\begin{lemma}
Suppose that $\sG$ is a Green's-current as in Definition \ref{def:Green-current} and
\[
\sG(B,v)=\sum_{\mu\in L^*}E_\mu(B)e^{2\pi i\langle v,\mu\rangle}
\]
its Fourier-series. If we assume that $E_0=0$, then the differential equation $\nabla(\sG)=\delta_{\sP\times\{0\}}\vol-\vol$ amounts to 
\begin{align}\label{eq:diff-eq}
&&dE_\mu+(2\pi i\mu-\kappa)E_\mu=\vol&&\mbox{ for all $\mu\neq 0$},&&
\end{align}
i.e.,
\begin{align}
&&(2\pi i\mu-\kappa)E_\mu^0=\vol&&\mbox{and}&& dE_\mu^a+(2\pi i\mu-\kappa)E_\mu^{a+1}=0.&&
\end{align}
Here we view $\mu\in V^*\subset \Lambda^\cdot V^*\widehat{\otimes}R$ as an $R$-valued $1$-form,
so that $(2\pi i\mu-\kappa)\in  \Lambda^\cdot V^*\widehat{\otimes}R$.
\end{lemma}
\begin{proof}
Immediate calculation using $\nabla=\nabla_0-\kappa$
and the fact that the Fourier series of $\delta_{\sP\times\{0\}}\vol$ is $\sum_{\mu\in L^*}
e^{2\pi i\langle v,\mu\rangle}\vol$.
\end{proof}
We will now forget the fact that $\mu$ comes from the
lattice $L^*$ and try to find a natural solution of \eqref{eq:diff-eq} for any $0\neq\mu\in V^*$.
For this we consider the half-space $V_{\mu>0}:=\{v\in V\mid \langle v,\mu\rangle >0\}$. We 
consider $B$ as an isomorphism $B:V\isom V^*$, so that we have a map
\begin{align}
&&&v_\mu:\sP\to V_{\mu>0}&& B\mapsto B^{-1}(\mu).&&&
\end{align}
We will construct $E_\mu$ as the $v_\mu$-pull-back of a natural $n-1$-form $E_{(\mu)}$ on $V_{\mu>0}$.

Define the commutative DG-algebra $\cA:=\Omega^\cdot(V)\otimes \Lambda^\cdot V^*$ with differential 
$d(\omega\otimes \xi)=d\omega\otimes \xi$. On $\cA$ we have the derivation $\theta$ of degree $-1$,
which is zero on $\Omega^\cdot(V)$ and maps $\mu\in V^*\subset \Lambda^\cdot V^*$ to the linear
function $\mu_V\in \sC^\infty(V)$ with $\mu_V(v):=\mu(v)$.
The DG-algebra $\cA$ contains the subalgebra 
\[
{\Lambda}^\cdot(V^*\oplus V^*)\isom {\Lambda}^\cdot V^*\otimes {\Lambda}^\cdot V^*\subset \Omega^\cdot(V)\otimes {\Lambda}^\cdot V^*
\]
and we let $\Delta:\Lambda V^*\to {\Lambda}^\cdot V^*\otimes {\Lambda}^\cdot V^*$ be the 
algebra homomorphism induced by the diagonal map $V^*\to V^*\oplus V^*$. 

Let $\ell_1,\ldots,\ell_n$ be a basis of $L$
and $\mu_1,\ldots,\mu_n$ be the dual basis of $V^*$. Then
\[
\vol:=\mu_1\wedge\ldots\wedge \mu_n\otimes \ell_1\wedge\ldots\wedge\ell_n\in {\Lambda}^n V^*\otimes\lambda
\]
and $\kappa=\sum_{j=1}^nd\mu_{j,V}\otimes \ell_j\in \Omega^1(V)\widehat{\otimes} R$.
\begin{definition}
We let $\psi:=\Delta(\vol)\in \cA^n\otimes\lambda$ and write $\psi=\sum_{a=0}^n\psi^a$ with
$\psi^a\in \Omega^a(V)\otimes \Lambda^{n-a}V^*\otimes\lambda$. Then we define
\[
\nu^a:=\theta(\psi^a)\in \Omega^a(V)\otimes {\Lambda}^{n-1-a}V^*\otimes\lambda.
\]
\end{definition}
We note that the forms $\nu^a$ have the following explicit
description. Let $\omega_i:=d\mu_{i,V}\in\Omega^1(V)$,
$\omega_I:=\Lambda_{i\in I}\omega_i$ for any 
subset $I\subset \{1,\ldots,n\}$ and define
similarly $\mu_I$. Then 
\[
\nu^a=\sum_{|I|=a}\sum_{j=1}^{n}\mu_{j,V}\omega_I\otimes \mu_{I^c\setminus\{j\}},
\]
where $I^c$ is the complement of $I$. 
The forms $\nu^a$ have the following properties:
\begin{lemma}\label{lemma:nu-formulae}
For $\xi\in V^*$ one has the formulae
\begin{align*}
d\nu^a&=(a+1)\psi^{a+1}\\
\xi\wedge \psi^a&=-d\xi_V\wedge \psi^{a-1}\\
\xi\wedge \nu^a&=\xi_V\psi^a-d\xi_V\wedge \nu^{a-1}
\end{align*}
In particular, if one writes $\kappa_V:=\sum_{j=1}^n \mu_{j,V}\otimes \ell_j$, so that $d\kappa_V=\kappa$, one has
\[
\kappa\wedge \nu^a =\kappa_V\psi^a-d\kappa_V\wedge\nu^{a-1}.
\]
\end{lemma}
\begin{proof}
For a form 
$\omega\in \cA$ denote by $\omega^a\in \Omega^a(V)\otimes \Lambda^{n-a}V^*$ its $a$-part. 
For $\xi\in V^*$ one has $\Delta(\xi)=d\xi_V+\xi$ and hence
$d\theta(\Delta(\xi)^1)=0$ and $d\theta(\Delta(\xi)^0)=d\xi_V=\Delta(\xi)^1$. 
Write $\vol_n=\vol_{n-1}\wedge\mu_n$. Then
\[
\Delta(\vol_n)^a=\Delta(\vol_{n-1})^{a-1}\wedge \Delta(\mu_n)^1+
\Delta(\vol_{n-1})^{a}\wedge \Delta(\mu_n)^0.
\]
Applying $d\theta$ and induction on $n$ gives
\[
d\nu^a=a\Delta(\vol_{n-1})^{a}\wedge \Delta(\mu_n)^1+
(a+1)\Delta(\vol_{n-1})^{a+1}\wedge \Delta(\mu_n)^0+
\Delta(\vol_{n-1})^{a}\wedge \Delta(\mu_n)^1.
\]
This shows the first equation. The second follows from 
$\Delta(\xi)\wedge \Delta(\vol)=\Delta(\xi\wedge\vol)=0$
and $\Delta(\xi)=d\xi_V+\xi$ and the third by applying $\theta$
to it. The formula for $\kappa$ follows from the third equation
using the explicit formulae for $\kappa$ and $\kappa_V$.
\end{proof}
Write $\cA_{(\mu)}:=\Omega^\cdot(V_{>\mu})\otimes \Lambda^\cdot V^*$, then $\mu_V$ is invertible in
$\cA_{(\mu)}$. The element $\kappa_V:=\sum_{j=1}^n \mu_{j,V}\otimes \ell_j\in \sC^\infty(V)\widehat{\otimes}R$ is topologically
nilpotent, so that $\mu_V-\kappa_V$ is invertible in $ \sC^\infty(V_{>\mu})\widehat{\otimes}R$. 
Define
\begin{align}
&&E_{(\mu)}^a:=(-1)^a a!(2\pi i\mu_V-\kappa_V)^{-a-1}\nu^a&& E_{(\mu)}:=\sum_{a=0}^{n-1}E_{(\mu)}^a.&&
\end{align}
\begin{lemma}\label{lemma:diff-eq-for-Emu}
The formulae
\begin{align*}
(2\pi i\mu-\kappa)E_{(\mu)}^0&=\psi^0=\vol\\
d E_{(\mu)}^a+(2\pi i\mu-\kappa)E_{(\mu)}^{a+1}&=0
\end{align*}
hold. In particular, $E_\mu:=v_\mu^*E_{(\mu)}$ satisfies the differential equation \eqref{eq:diff-eq}.
Moreover, for $\gamma\in \Gamma$ one has
\[
\gamma^*E_\mu=E_{\mu\circ \gamma^{-1}}.
\]
\end{lemma}
\begin{proof}
From Lemma \ref{lemma:nu-formulae} we have
\begin{align*}
(2\pi i\mu-\kappa) \nu^{a+1}&=(2\pi i\mu_V-\kappa_V)\psi^{a+1}-d(2\pi i\mu_V-\kappa_V)\wedge \nu^{a}\\
d((2\pi i\mu_V-\kappa_V)^{-a-1}\nu^a)&= (a+1)\left( (2\pi i\mu_V-\kappa_V)^{-a-1}\psi^{a+1}
-(2\pi i\mu_V-\kappa_V)^{-a-2}  {d(2\pi i\mu_V-\kappa_V)}\nu^a\right)
\end{align*}
which show that the differential equations are satisfied. For the action of $\gamma$
note that $v_\mu\circ\gamma=\gamma^{-1}\circ v_{\mu\circ\gamma^{-1}}$. As $\vol$ and $\theta$ are
$\Gamma$-invariant one has $(\gamma^{-1})^*\nu^a=\nu^a$. 
The map  $\kappa_V:V\to R$ is the canonical inclusion and obviously $\Gamma$-invariant. 
Therefore $(\gamma^{-1})^*E_{(\mu)}=E_{(\mu\circ\gamma^{-1})}$ and the formula follows.
\end{proof}
\begin{lemma}\label{lemma:Emu-explicit}
Let $E_\mu^a:=v_\mu^*(E^a_{(\mu)})$ and consider $B^{-1}:V^*\isom V$
as bilinear form on $V^*$, then one has explicitly
\[
E_\mu^a=(-1)^a a!\frac{v_\mu^*(\nu^a)}{(2\pi iB^{-1}(\mu,\mu)-B^{-1}(\mu))^{a+1}}
=(-1)^a\frac{(k+a)!}{k!}\sum_{k\ge 0}
\frac{ B^{-1}(\mu)^{\otimes k}}{(2\pi iB^{-1}(\mu,\mu))^{k+a+1}}v_\mu^*(\nu^a)
\]
(where we let $B^{-1}(\mu)^{\otimes 0}:=1$) and
\[
v_\mu^*(\nu^a)=\sum_{|I|=a}\sum_{j=1}^{n}B^{-1}(\mu_{j},\mu)\Lambda_{i\in I}dB^{-1}(\mu_i,\mu)\otimes \mu_{I^c\setminus\{j\}}\otimes \ell_1\wedge\ldots\wedge\ell_n.
\]
\end{lemma}
\begin{proof}
Direct computation.
\end{proof}
We are going to show that the series 
$\sum_{\mu\in L^*\setminus\{0\}}E^a_\mu e^{2\pi i\langle v,\mu\rangle} $ defines a current on $\sP\times T$, which is
smooth on $\sP\times (T\setminus \{0\})$. 
The following proof is due to Levin (see the Appendix of \cite{Blottiere1}). Let
$P:\sP\times V\to \CC$ be a $\sC^\infty$ function, which is
a homogeneous polynomial of degree $g$ in the $V$-variables. We consider
the series of distributions
\[
\sK_s(B,v,P):=\sum_{\mu\in L^*\setminus \{0\}}
\frac{ P(B,\mu)}{ B^{-1}(\mu,\mu)^{s+g/2}}e^{2\pi i\langle v,\mu\rangle}
\]
and we are interested in the convergence and the analyticity in
$s$. We have the following result:
\begin{theorem}\label{thm:Epstein-zeta}
Let $v \neq 0$ then $\sK_s(B,v,P)$ is a smooth distribution for
all $s\in \CC$.
\end{theorem}
\begin{proof} We give the essential steps of the proof.

The first step is to show that the series 
$\sK_s(B,v,P)$ defines a (tempered) distribution on $\sP\times T$
for all $s\in \CC$. 
We may assume that $B$ varies in a compact subset of $\sP$.
A Fourier series defines a distribution, if
the coefficients grow less than a polynomial of fixed 
degree $N\ge 0$. But $\frac{ B^{-1}(\mu)^{\otimes k}}{(2\pi iB^{-1}(\mu,\mu))^{k/2+s}}$ satisfies this requirement if 
$s\ge -N$. 

The second step is to remark that the map $s\mapsto \sK_s(B,v,P)$
is analytic. This follows because for each test function
$\psi$ the series $\sK_s(B,v,P)(\psi)$ converges absolutely and uniformly 
on every compact subset of $\CC$ (same proof as for Dirichlet series, one also has to use that weakly analytic functions
with values in the dual of a Frechet space are actually analytic).

Next we note that $\sK_s(B,v,P)$ converges as a sequence of
functions absolutely and uniformly for $\Re(s)\ge n/2+\epsilon$
with $\epsilon>0$. The resulting analytic function on 
the half plane $\Re(s)>n/2$ can be analytically continued with
the standard procedure known from the analytic continuation
of the zeta functions: One writes $\sK_s(B,v,P)$ as the Mellin-transform of a theta series as in \cite[Chapter I, Paragraph 5]{Siegel}
and uses the Poisson summation formula to obtain the analytic
continuation $\widetilde{\sK}_s(B,v,P)$ of $\sK_s(B,v,P)$. To see that the function $\widetilde{\sK}_s(B,v,P)$ has no poles
one uses \cite[Chapter I,5, Theorem 3]{Siegel}. Note that
our polynomial function $P$ is homogeneous so that its value
at $v=0$ is $0$ if the degree $g>0$. For $g=0$ the polynomial 
is constant and it is here that the assumption $v\neq 0$ enters
to guaranty that $\widetilde{\sK}_s(B,v,P)$ has no pole.

Finally we remark that the principle of analytic continuation
holds for tempered distributions, so that we can conclude
that $\widetilde{\sK}_s(B,v,P)=\sK_s(B,v,P)$ for all $s\in \CC$.
This shows the assertion of the theorem.
\end{proof}
The next corollary finishes the proof of Theorem 
\ref{thm:existence-green-currents}.
\begin{corollary}\label{cor:Epstein-zeta}
The series 
\[
\sG(B,v)=\sum_{a=0}^{n-1}\sum_{\mu\in L^*\setminus \{0\}}E_\mu^a e^{2\pi i \langle v,\mu\rangle}
\] 
defines a $R\otimes\lambda$-valued, $\Gamma$-invariant current on $\sP\times T$,
which is smooth on $\sP\times (T\setminus\{0\})$ and
satisfies the differential equation
\[
\nabla(\sG)=\delta_{\sP\times \{0\}}\vol-\vol.
\] 
In particular, $\sG$ is a Green's-current.
\end{corollary}
\begin{proof}
We have 
\begin{equation*}
\sum_{\mu\in L^*\setminus \{0\}}
\frac{ B^{-1}(\mu)^{\otimes k}e^{2\pi i \langle v,\mu\rangle}}{(2\pi iB^{-1}(\mu,\mu))^{k+a+1}}v_\mu^*(\nu^a)
=
\frac{1}{(2\pi i)^{k+a1}}
\sum_{k\ge 0}\sK_{(k+a+1)/2}(B,v,P)\omega
\end{equation*}
where $\omega\in \Omega^a(\sP)\otimes \Lambda^{n-1-a}V^*$
is a smooth differential form, which does not depend on $\mu$
and $P$ is a polynomial of degree $k+a+1$ in the $V$-variables.
Hence, by Theorem \ref{thm:Epstein-zeta} the left 
hand side defines a current on $\sP\times T$, which is 
smooth on $\sP\times (T\setminus\{0\})$. By Lemma
\ref{lemma:diff-eq-for-Emu} we have
$\gamma^*(E_\mu e^{2\pi i\langle v,\mu\rangle})=
E_{\mu\circ\gamma^{-1}} e^{2\pi i\langle v,\mu\circ\gamma^{-1}\rangle}$, which shows that $\sG(B,v)$ is
$\Gamma$-invariant. The differential equation is an immediate
consequence of Lemma \ref{lemma:diff-eq-for-Emu}.
\end{proof}
%
%
\section{Applications to $L$-values of totally real fields and Eisenstein cohomology
of Hilbert modular varieties}
%
%
We discuss the relation between the topological polylogarithm
and special values of partial $L$-functions of totally real 
fields. This is due to Nori and Szcech but we need the
explicit formulae  for the $p$-adic interpolation. The second
application shows the relation of the topological polylogarithm 
to Eisenstein cohomology for Hilbert modular varieties. This
is a new result due to Graf and the detailed relationship 
will appear in his thesis \cite{Graf}. We discuss here the 
$p$-adic interpolation of his construction.
%
\subsection{Values of  partial $L$-functions of totally real fields}
%
In this section $F$ will be a totally real field of degree
$n$ over $\QQ$ and ring of integers $\cO_F$. Let
$L\subset F\otimes \RR$ be a fractional ideal and $h$ be
an element which is non-zero and torsion in $T(L)=F\otimes \RR/L$.
We define
\[
\cO_h^{+,\times}:=\{u\in \cO_F^\times\mid uh\equiv h\bmod L
\mbox{ and }u\mbox{ totally positive}\}
\]
We consider the partial zeta function for $\Re(s)>1$
\begin{equation}
\zeta(h,L,s):=\sum_{\alpha\in (h+L)^+/\cO_h^{+,\times}}\N\alpha^{-s}.
\end{equation}
where $(h+L)^+$ are the totally positive elements in $h+L$.
A  sign
character $\varepsilon:(F\otimes \RR)^\times \to \{\pm 1\}$ 
is a character, which is trivial on $(F\otimes \RR)^{+,\times}$
the connected component of $1$ in $(F\otimes \RR)^\times$.
Each $\varepsilon$ is the product of $|\varepsilon|$ sign characters of embeddings of $F$. 
We consider also more generally the partial zeta functions
\begin{equation}
\zeta(\varepsilon, h, L,s):=\sum_{\alpha\in (h+L)/\cO_h^{+,\times}}\frac{\varepsilon(\alpha)}{|\N\alpha|^s}
\end{equation}
and we have the identity
\begin{equation}
\sum_{\varepsilon}\zeta(\varepsilon, h, L,s)=2^n \zeta(h,L,s).
\end{equation}
\begin{remark}
Let $\frf$ and $\frb$ be two coprime integral ideals of $F$ and set
$L=\frf\frb^{-1}$. Then $\N\frb^{-s}\zeta(1,\frf\frb^{-1},s)$ is the
partial zeta function $\zeta(\frb, \frf,s)$ considered by
Siegel in \cite{Siegel-Fourier}.
\end{remark}
\begin{proposition} \label{prop:functional-equation}
The function $\zeta(\varepsilon, h, L,s)$ admits
an  analytic continuation to $\CC$ and satisfies the
functional equation 
\[
\zeta(\varepsilon, h, L,1-s)=
(\cos(\pi (s+1)/2))^{|\varepsilon|}
(\cos(\pi s/2))^{n-|\varepsilon|}
\frac{2^n i^{|\varepsilon|}\Gamma(s)^n}{(2\pi)^{ns}\vol(L)}
\sum_{\mu\in L^*\setminus\{0\}/\cO_h^{+,\times}}\frac{\varepsilon(\mu)e^{2\pi i\langle h,\mu\rangle}}{|\N\mu|^s}
\]
where $L^*=\Hom_\ZZ(L,\ZZ)$ is the dual lattice. 
\end{proposition}
\begin{proof}
This is a standard result. A sketch of the proof can
be found in \cite{Siegel-Fourier} (for $h=1$). The case of 
general $h$ is the same. Alternatively the result can be
deduced from \cite[Theorem 3.12]{Deligne-Ribet}.
\end{proof}
\begin{corollary}\label{cor:L-value}
Let $\sgn^{k+1}$ be the sign character $\sgn^{k+1}(\mu):=
\frac{N(\mu)^{k+1}}{|N(\mu)|^{k+1}}$. Then for
any integer $k\ge 0$, the value $\zeta(\varepsilon, h, L,-k)$ is $0$ except for
$\varepsilon=\sgn^{k+1}$. In particular,
$\zeta(\sgn^{k+1}, h, L,-k)=2^n\zeta( h, L,-k)$ and
one has
\[
\zeta(h, L,-k)=
\frac{(k!)^n}{(2\pi i)^{n(k+1)}\vol(L)}
\sum_{\mu\in L^*\setminus\{0\}/\cO_h^{+,\times}}\frac{e^{2\pi i\langle h,\mu\rangle}}{\N\mu^{k+1}}.
\]
\end{corollary}
\begin{proof}
This is an easy consequence of Proposition \ref{prop:functional-equation} and the location of the zeroes of the functions
$(\cos(\pi (s+1)/2))^{|\varepsilon|}$ and 
$(\cos(\pi s/2))^{n-|\varepsilon|}$.
\end{proof}

%
\subsection{The evaluation map}
%

\label{section:L-values}
We keep the notation of the previous section, i.e.,
$F$  is a totally real field of degree
$n$ over $\QQ$ with ring of integers $\cO_F$, 
$L\subset F\otimes \RR$ is a fractional ideal  and we
consider the torus $T:=F\otimes \RR/L$.

Let $\cO_F^{+,\times}$ 
be the group of totally positive units in $\cO_F$.
Note that this is a free abelian group of rank $n-1$.
Let $D\subset T^{(A)}$ be a finite non-empty set of
torsion points, $t\in T^{(A)}$ and $\alpha\in (A[D]^0)^\Gamma$.
We consider the Eisenstein class
\[
\Eis_\alpha(t)\in H^{n-1}(\Gamma,R\otimes\lambda(L))
\]
from Definition \ref{def:Eisenstein-class}, where 
$\Gamma\subset\cO_F^{+,\times}$ is the stabilizer of $D$ and $t$. Note that $\Gamma$ acts through the norm and hence trivially on
$\lambda$. The cap-product 
with $H_{n-1}(\Gamma,\ZZ)$ induces a homomorphism
\begin{equation}\label{eq:cap-product}
H^{n-1}(\Gamma,R\otimes\lambda(L))\otimes H_{n-1}(\Gamma,\ZZ)
\to H_0(\Gamma,R\otimes \lambda)=R_\Gamma\otimes \lambda
\end{equation}
where $R_\Gamma$ are the $\Gamma$-coinvariants.
For the actual evaluation we choose coordinates for $T$, which
at the same time allow us to trivialize $\lambda(L)\otimes\RR$ and
to give a generator for $ H_{n-1}(\Gamma,\ZZ)$.

Let $\{\tau_1,\ldots,\tau_n\}$
be the different embeddings of $F$ into $\RR$, so that we have
an isomorphism
$F\otimes\RR\isom \RR^n$. On $\RR^n$ we use the standard orientation. 
For each $x\in F\otimes \RR$ we write $x_i:=\tau_i(x)$.

If we identify $\lambda=H_n(L,\ZZ)\isom H_n(F\otimes\RR/L,\ZZ)$,
then the fundamental class of $F\otimes\RR/L$ provides us
with a generator of $\lambda$.

Let $(F\otimes\RR)^{+,\times}$ be the totally positive 
and invertible elements in $F\otimes \RR$. This is the connected
component of $1$ in $(F\otimes\RR)^\times$. The norm of $F/\QQ$ 
defines a homomorphism $\N:(F\otimes \RR)^\times\to \RR^\times$ 
and we denote by
\[
(F\otimes \RR)^1:=\ker((F\otimes\RR)^{+,\times}\xrightarrow{\N} \RR^{+,*})
\]
the subgroup of elements of norm $1$. Then $\Gamma \subset\cO_F^{+,\times}\subset (F\otimes \RR)^1$ and one has a canonical isomorphism
\[
H_{n-1}(\Gamma,\ZZ)\isom H_{n-1}((F\otimes \RR)^1/\Gamma,\ZZ)
\]
with the homology of $(F\otimes \RR)^1/\Gamma$. The logarithm
$\log:(F\otimes\RR)^{+,\times}\xrightarrow{\isom}F\otimes \RR\isom \RR^n$
induces an orientation on $(F\otimes\RR)^{+,\times}$. Using
the standard orientation on $\RR^{+,\times}$ this induces
also an orientation on $(F\otimes \RR)^1$.
We use the fundamental class of $(F\otimes \RR)^1/\Gamma$
as a generator of $H_{n-1}(\Gamma,\ZZ)$. 
\begin{definition}\label{def:evaluation}
With the above
notations and generators we define
the \emph{evaluation map} to be the homomorphism
induced by \eqref{eq:cap-product}
\begin{equation*}
\ev:H^{n-1}((F\otimes \RR)^1/\Gamma, R)\to
R_\Gamma.
\end{equation*}
\end{definition}
Note that $\ev$ is defined for any coefficient ring $A$.
In the case $A=\RR$ or $\CC$ the isomorphism $F\otimes \RR\isom\RR^n$ induces 
\[
\lambda(L)\otimes \RR\isom\lambda(\ZZ^n)\otimes\RR
\]
and we define $\vol(L)\in \RR$, such that 
$\vol(L)\lambda(\ZZ^n)$ corresponds to the lattice
$\lambda(L)$ under this isomorphism.
Then the evaluation is given  
explicitly  by 
\begin{equation}
\ev(\eta)=\vol(L)^{-1}\int_{(F\otimes \RR)^1/\Gamma} \eta
\end{equation}
for a differential form $\eta\in H^{n-1}((F\otimes \RR)^1/\Gamma, R\otimes \lambda(L))$.

We give a more explicit description of $(R_\CC)_\Gamma$. 
The isomorphism $L_\RR\isom\RR^n$ allows us to identify 
$R_\CC\isom \widehat{\Sym}L_\CC$ with the power series ring $\CC[[z_1,\ldots,z_n]]$.
The action of $\Gamma\subset\cO_F^{+,\times}$ on $L\otimes \RR$
decomposes into a direct sum of homomorphisms
$\tau_i:\Gamma\to \RR^\times$, such that 
$u\in \Gamma$ acts as $\tau_i(u)z_i$ on $z_i$. 
\begin{lemma}\label{lemma:coinvariants}
Let $w:=z_1\cdots z_n$ be the product of the $z_i$'s, then 
\[
R_\CC^\Gamma=\CC[[w]]\subset \CC[[z_1,\ldots,z_n]]=R_\CC
\]
and the projection $p_\Gamma:R_\CC\to (R_\CC)_\Gamma$
induces an isomorphism $(R_\CC)^\Gamma\isom (R_\CC)_\Gamma$. 
\end{lemma}
\begin{proof}
On each monomial $z_1^{k_1}\cdots z_n^{k_n}$ the element
$u\in \Gamma$ acts
via $\tau(u)^{k_1}\cdots \tau_n(u)^{k_n}$, so that the action
of $\Gamma\otimes\RR$ on $R_\CC$ is semi-simple. In particular, $R_\CC^\Gamma=R_\CC^{\Gamma\otimes\RR}$
is a direct summand. Moreover, as each trivial $\Gamma$-representation has to factor through the norm, $\Gamma$ acts trivially exactly
on $w^k$ for integers $k\ge 0$.
\end{proof}
\begin{remark}\label{rem:coinvariants}
Let $A$ be any $\QQ$-algebra and let $\Lambda^\cdot:=\Lambda^\cdot\Hom(\Gamma,\QQ)=H^\cdot(\Gamma,\QQ)$. Then the projection $p_\Gamma:R_A\to (R_A)_\Gamma$ yields
isomorphisms
\[
H^p(\Gamma,R_A)\isom H^p(\Gamma,(R_A)_\Gamma)=(R_A)_\Gamma\otimes \Lambda^p.
\]
\end{remark}
%
\subsection{The topological polylogarithm and $L$-values of totally real fields}
%
Theorem \ref{thm:green-currents} implies that the class of $\Eis_\alpha(t)$ is represented by
$\sum_{d\in D}\alpha(d)(t-d)^*\sG$, 
where we consider $t-d$ as a torsion section of the torus family
$((F\otimes \RR)^1\times T)/\Gamma \to (F\otimes \RR)^1/\Gamma$. 
Note that for any torsion section $h$
\[
h^*\sG\in H^{n-1}((F\otimes \RR)^1/\Gamma, R_\CC\otimes \lambda(L))
\]
because $\nabla(h^*\sG)=h^*(\delta_{(F\otimes \RR)^1/\Gamma\times\{0\}}\vol-\vol)=0$.
We can now formulate the main result in this section.
\begin{theorem}\label{thm:G-evaluation} Let $h\in T=F\otimes\RR/L$ be a non-zero torsion section and let $\Gamma=\cO_h^{+,\times}$ be the stabilizer of $h$ in $\cO_F^{+,\times}$. Identify $(R_\CC)_\Gamma\isom \CC[[w]]$ as 
in Lemma \ref{lemma:coinvariants}. Then
one has 
\[
\ev(h^*\sG)=(-1)^{n-1}\sum_{k\ge 0}\zeta(h,L,-k)\frac{w^k}{(k!)^n}.
\]
Equivalently, using the isomorphism $\widehat{\Sym}L_\CC\isom\widehat{\TSym}L_\CC$, we get
\[
\exp^*_k\circ\ev(h^*\sG)=\zeta(h,L,-k)z_1^{[k]}\cdots z_n^{[k]}\in 
\TSym^kL_\CC.
\]
\end{theorem}
\begin{proof}
From Corollary \ref{cor:Epstein-zeta} we know that
$
\sG(B,v)=\sum_{a=0}^{n-1}\sum_{\mu\in L^*\setminus \{0\}}E_\mu^a e^{2\pi i \langle v,\mu\rangle}
$ 
and by definition $h^*E_\mu^a=0$ for $a\neq n-1$. 

For the evaluation we use the following explicit embedding of
$(F\otimes \RR)^1\to \sP$. For $q=(q_1,\ldots,q_n)\in (F\otimes\RR)^1$ we consider the form $B_q\in \sP$ on $\RR^n$, defined by
\begin{equation}
B_q(v,w):=\sum_{j=1}^nq_j^{-1} v_jw_j.
\end{equation}
Then the map $v_\mu:\sP\to V$ is given by $v_\mu(B_q)=(q_1\mu_1,\ldots,q_n\mu_n)$ and
writing $R=\CC[[z_1,\ldots,z_n]]$ the map $\kappa_V:V\to R$ is given by 
$\kappa_V(v)=\sum_{j=1}^n v_jz_j$.
We want to compute the integral
\begin{align}\label{eq:integral}
\begin{split}
\ev(h^*\sG)&=\vol(L)^{-1}\int_{(F\otimes \RR)^1/\Gamma}h^*\sG\\
&=\frac{(-1)^{n-1}}{\vol(L)}\sum_{\mu\in L^*\setminus\{0\}/\Gamma}e^{2\pi i\langle h,\mu\rangle}
\int_{(F\otimes \RR)^1}\frac{ (n-1)!v_\mu^*(\nu^{n-1})}{(\sum_{j=1}^n 2\pi i\mu_j^2q_j-\mu_jq_jz_j)^n}.
\end{split}
\end{align}
Using $\N(q)=q_1\cdots q_n=1$ we get
\[
v_\mu^*(\nu^{n-1})\mid_{(F\otimes \RR)^1}=
\N(\mu)\sum_{j=1}^n(-1)^{j-1}d\log q_1\wedge\ldots\wedge
\widehat{d\log q_j}\wedge\ldots d\log q_n.
\]
Let $y_1,\ldots,y_n$ be the coordinate functions of $\RR^n$
and let $t:=(\N y)^{1/n}$, so that $y_j=tq_j$. 
Then $d\log t=\frac{1}{n}\sum_{j=1}^nd\log y_j$, which gives
\begin{equation}\label{eq:differential-form}
\N\mu\frac{dy_1}{y_1}\wedge\ldots
\wedge \frac{dy_n}{y_n}=\frac{dt}{t}v_\mu^*(\nu^{n-1})\mid_{(F\otimes \RR)^1}.
\end{equation}
We write 
\[
\frac{ (n-1)!}{(\sum_{j=1}^n 2\pi i\mu_j^2q_j-\mu_jq_jz_j)^n}
=
\int_{\RR^{+,\times}}e^{-t(\sum_{j=1}^n 2\pi i\mu_j^2q_j-\mu_jq_jz_j)}t^n\frac{dt}{t}
\]
and we substitute this and \eqref{eq:differential-form} into \eqref{eq:integral}. Using
the exact sequence $0\to (F\otimes \RR)^1\to (F\otimes \RR)^{+,\times}\xrightarrow{\N} \RR^{+,\times}\to 0$, we have to compute the  integral
\begin{align*}
\int_{(F\otimes \RR)^{+,\times}}e^{-(\sum_{j=1}^n 2\pi i\mu_j^2q_j-\mu_jq_jz_j)}\N y\frac{dy_1}{y_1}\wedge\ldots
\wedge \frac{dy_n}{y_n}&=
\prod_{j=1}^n\int_{\RR^{+,\times}}
e^{-y_j\mu_j(2\pi i\mu_j-z_j)}y_j\frac{dy_j}{y_j}\\
&=\prod_{j=1}^n\frac{1}{\mu_j(2\pi i\mu_j-z_j)}\\
&=\frac{1}{\N\mu^2}\sum_{\ell\ge 0}
\frac{1}{(2\pi i)^{\ell+n}}\prod_{\ell_1+\ldots+\ell_n=\ell}
\frac{z_1^{\ell_1}\cdots z_n^{\ell_n}}{\mu_1^{\ell_1}\cdots
\mu_n^{\ell_n}}.
\end{align*}
If we apply the projection $p_\Gamma:\CC[[z_1,\ldots,z_n]]\to \CC[[w]]$ only the monomials for $\ell=nk$ of the form
$(\frac{w}{\N\mu})^k$ survive and we get
\[
p_\Gamma\int_{(F\otimes \RR)^1/\Gamma}\frac{ (n-1)!v_\mu^*(\nu^{n-1})}{(\sum_{j=1}^n 2\pi i\mu_j^2q_j-\mu_jq_jz_j)^n}=
\sum_{k\ge 0}\frac{w^k}{(2\pi i)^{n(k+1)}\N\mu^{k+1}}.
\]
This gives 
\begin{align*}
\ev(h^*\sG)&=(-1)^{n-1}\vol(L)^{-1}\sum_{k\ge 0}
\left(\sum_{\mu\in L^*\setminus\{0\}/\Gamma}
\frac{e^{2\pi i\langle h,\mu\rangle}}{\N\mu^{k+1}}
\right)
\frac{w^k}{(2\pi i)^{n(k+1)}}\\
&=(-1)^{n-1}\sum_{k\ge 0}\zeta(h,L,-k)\frac{w^k}{(k!)^n}
\end{align*}
where in the last line we have used Corollary \ref{cor:L-value}.
The formula for $\exp^*_k(\ev(h^*\sG))$ follows from the fact that
$w^k=(z_1\cdots z_n)^k\mapsto (k!)^nz_1^{[k]}\cdots z_n^{[k]}$ under the isomorphism $\Sym^kL_\CC\isom \TSym^kL_\CC$.
\end{proof}
From the theorem we immediately get all the known rationality,
integrality and $p$-adic interpolation properties of 
$\zeta(h,L,s)$. For this we use the following principle:
For any subring $A\subset \CC$ consider the natural inclusion
$R_A\subset R_\CC$. Then we have a commutative diagram
\begin{equation}
\begin{split}
\begin{CD}
H^{n-1}(\Gamma,R_A\otimes \lambda)@>\ev >> (R_A)_\Gamma\\
@VVV@VVV\\
H^{n-1}(\Gamma,R_\CC\otimes \lambda)@>\ev >> (R_\CC)_\Gamma
\end{CD}
\end{split}
\end{equation}
and any class coming from $H^{n-1}(\Gamma,R_A\otimes \lambda)$ has to have coefficients in $A$ under the evaluation map.

To express the next results also in more classical notation,
let $\frf, \frb$ be coprime integral ideals and
$L:=\frf\frb^{-1}$. Then Siegel's zeta function is defined to be
\begin{equation}
\zeta(\frb,\frf,s):=\N\frb^{-s}\zeta(1,\frf\frb^{-1},s).
\end{equation}
\begin{corollary}[Klingen-Siegel]\label{cor:Klingen-Siegel}
Let $h\in F\otimes\RR/L$ be a non-zero torsion point, then for
$k\ge 0$ one has $\zeta(h,L,-k)\in \QQ$. In particular,
\[
\zeta(\frb,\frf,-k)\in \QQ.
\]
\end{corollary}
\begin{proof}
Recall from Theorem \ref{thm:green-current-variant} that
the class from Definition \ref{def:Eisenstein-variant}
\[
\Eis(h)=\contr(h^*\pol)\in H^{n-1}(\Gamma,R_\QQ\otimes\lambda)
\]
is represented by $\contr (h^*(\varpi{\sG}+\eta)=)=\contr(\varpi (h^*\sG))$. By definition of
$\mult$ in \eqref{eq:mult-def} we have 
$ \varpi (h^*\sG)=\mult (h^*\sG)$, so that $\Eis(h)$ is 
represented by $h^*\sG$ (recall that $\contr\circ\mult=\id$).
It follows from Theorem \ref{thm:G-evaluation} and the above
principle that 
\[
\exp^*_k(\ev(\Eis(h))=(-1)^{n-1}\zeta(h,L,-k)z_1^{[k_1]}\cdots z_n^{[k_n]}
\]
has coefficients in $\QQ$.
\end{proof}
 
\begin{corollary}[Deligne-Ribet \cite{Deligne-Ribet}]\label{cor:Deligne-Ribet}
Let $\frc$ be an integral ideal coprime to
$\frf\frb^{-1}$. Then for $k\ge 0$ one has 
\[
(\N\frc)^{1+k}\zeta(\frb,\frf,-k)-\zeta(\frb\frc,\frf,-k)\in
\ZZ[\frac{1}{\N\frc}].
\]
\end{corollary}
\begin{proof} Let $L:=\frf\frb^{-1}$.
We consider the isogeny $[\frc]:L\to L\frc^{-1}$ of degree $\N\frc$
and let $A:=\ZZ[\frac{1}{\N\frc}]$, so that
$\ker[\frc]\subset T^{(A)}$. Then the function
$\alpha_{[\frc]}\in A[\ker[\frc]]^0$
defined in \eqref{def:alpha-varphi-def} gives an element
\[
\Eis_{\alpha_{[\frc]}}(1)\in H^{n-1}(\Gamma,R_{\ZZ[\frac{1}{\N\frc}]}\otimes \lambda).
\]
From Corollary \ref{cor:Eisenstein-comparison} we get 
\[
\Eis^k_{\alpha_{[\frc]}}(1)=
(\N\frc)\Eis^k(1)-{\Eis'}^k([\frc](1))
\]
and from the proof of Corollary \ref{cor:Klingen-Siegel}
we deduce that 
\[
\ev(\Eis^k_{\alpha_{[\frc]}}(1))=(-1)^{n-1}\left((\N\frc)\zeta(1,\frf\frb^{-1},-k)-
\zeta(1,\frf\frb^{-1}\frc^{-1},-k)\right) z_1^{[k_1]}\cdots z_n^{[k_n]}
\]
has a coefficient in $\ZZ[\frac{1}{\N\frc}]$.
Multiplying by the integer $\N(\frb\frc)^k$ gives the result
as stated.
\end{proof}
Finally, we deduce the $p$-adic interpolation of the
zeta values. This was proven independently by 
Barsky \cite{Barsky}, Cassou-Nogu\`es \cite{Cassou-Nogues} and 
Deligne-Ribet \cite{Deligne-Ribet}.

Fix a prime number $p$, an integral ideal $\frc$ prime to $p$ and let $A=\Zp$. Recall from \ref{prop:R-as-Iwasawa} that 
$R_{\Zp}$ is isomorphic to the 
Iwasawa algebra $\Meas(L_{\Zp},\Zp)$. 
Consider the polynomial function $w^k:L_{\Zp}\to \Zp$ defined by
the element $w^k\in \Sym^kL_{\Zp}^*$, which maps 
$a_1\ell_1+\ldots +a_n\ell_n$ to $(a_1\cdots a_n)^k$. Then the moment map 
\begin{equation}
\mom^k:(R_{\Zp})_\Gamma\to (\TSym^kL_{\Zp})_\Gamma\isom \Zp
z_1^{[k_1]}\cdots z_n^{[k_n]}
\end{equation}
maps $\mu\mapsto \mu(w^k)z_1^{[k_1]}\cdots z_n^{[k_n]}$.
We keep the lattice $L=\frf\frb^{-1}$ and consider the
function $\alpha_{[\frc]}$ for the isogeny $[\frc]:L\to L\frc^{-1}$
as defined in Definition \ref{def:alpha-varphi-def}.
\begin{corollary}[$p$-adic interpolation]\label{cor:p-adic-interpolation}
With the above notations the element 
\[
\ev(\Eis_{\alpha_{[\frc]}})\in (R_{\Zp})_\Gamma
\]
is a measure whose value on $w^k$ is  
$(-1)^{n-1}\left((\N\frc)\zeta(1,\frf\frb^{-1},-k)-
\zeta(1,\frf\frb^{-1}\frc^{-1},-k)\right)$.
\end{corollary}
\begin{proof}
This is immediate from the above remarks and Corollary \ref{cor:Deligne-Ribet}.
\end{proof}

%
\subsection{Eisenstein distributions and measures}
%

In this section we let $A=\Zp$ so that we can identify
$R=\Zp[[L_{\Zp}]]$ (see Proposition \ref{prop:R-as-Iwasawa}).
Denote by $T^{(p)}:=T^{(\Zp)}$ the subgroup of $T^{\tors}$ of elements
of order prime to $p$. For any  
$\Zp$-module $M$ we consider the
$\Zp$-module 
\[
M[[\Gamma\backslash L_{\Zp}]]:=\Meas(\Gamma\backslash L_{\Zp},M):=\prolim_r M[\Gamma\backslash (L/p^rL)]
\]
of $M$-valued distributions on $\Gamma\backslash L_{\Zp}$. Here we have set
$M[\Gamma\backslash (L/p^rL)]:=M\otimes_{\Zp}\Zp[\Gamma\backslash (L/p^rL)]$. Note that these are measures in the ordinary sense, if 
$M$ is a finitely generated $\Zp$-module, otherwise these are just
distributions. 
\begin{proposition}
For every function $g\in (\Zp[T^{(p)}\smallsetminus \{0\}]^0)^{\Gamma}$
there is an $H^{n-1}(\Gamma, R\otimes \lambda)$-valued distribution
\[
\mu^g_{L,\Eis}\in \Meas(\Gamma\backslash L_{\Zp},H^{n-1}(\Gamma, R\otimes \lambda))
\]
on $\Gamma\backslash L_{\Zp}$.
\end{proposition}
\begin{proof}
We are going to construct elements $\mu^g_{r,L,\Eis}\in M[\Gamma\backslash (L/p^rL)]$ in a compatible way. The distribution
$\mu^g_{r,L,\Eis}$ assigns to a $\Gamma$-invariant function
$f$ on $L/p^rL$ an element in $H^{n-1}(\Gamma, R\otimes \lambda)$.
This we define as follows. The isogeny $[p^r]:T_{p^rL}\to T_L$
associated to the inclusion $p^rL\subset L$ yields an isomorphism
$[p^r]:T_{p^rL}^{(p)}\isom T_L^{(p)}$, which allows to consider
the function $g\in (\Zp[T_L^{(p)}\smallsetminus \{0\}]^0)^{\Gamma}$
as an element $g_r\in (\Zp[T_{p^rL}^{(p)}\smallsetminus \{0\}]^0)^{\Gamma}$. Then Definition \ref{def:Eis-as-map} gives
an element
\[
\Eis(f\otimes g_r)\in H^{n-1}(\Gamma,\sLog_{p^rL}\mid_{L/p^rL}\otimes\lambda_{p^rL}),
\]
where we view $L/p^rL\subset T_{p^rL}$ is the kernel of the isogeny
$[p^r]$. The trace map associated to $[p^r]$ induces
\[
\Tr_{[p^r]}:H^{n-1}(\Gamma,\sLog_{p^rL}\mid_{L/p^rL}\otimes\lambda_{p^rL})\to
H^{n-1}(\Gamma,\sLog_ {L}\mid_{\{0\}}\otimes\lambda_{L})=
H^{n-1}(\Gamma, R\otimes\lambda)
\]
and we define 
\[
\mu^g_{r,L,\Eis}(f):=\Tr_{[p^r]}(\Eis(f\otimes g_r)).
\]
As $g_r=[p^r]^*(g)$ it follows from Corollary \ref{cor:Eis-trace-comp}
that the $\mu^g_{r,L,\Eis}(f)$ indeed define a distribution
on $\Gamma\backslash L_{\Zp}$.
\end{proof}
\begin{proposition} Recall that $R=\Zp[[L_{\Zp}]]$.
There is a canonical homomorphism
\[
H^{n-1}(\Gamma,\Zp[[L_{\Zp}]]\otimes \lambda)\to
\Meas(\Gamma\backslash L_{\Zp}, H^{n-1}(\Gamma,\Zp\otimes \lambda)).
\]
\end{proposition}
\begin{proof} 
The pairing between distributions and functions on $L/p^rL$ is
a map
\[
\Zp[L/p^rL]\times\Zp[L/p^rL] \to \Zp
\]
so that the
cup-product
defines a pairing
\[
H^0(\Gamma,  \Zp[L/p^rL])\otimes_{\Zp}
H^{n-1}(\Gamma,\Zp[L/p^rL]\otimes \lambda )\to 
H^{n-1}(\Gamma,\Zp\otimes\lambda)
\]
and hence a homomorphism
\begin{equation*}\label{eq:measure-map}
H^{n-1}(\Gamma,\Zp[L/p^rL]\otimes \lambda )\to
H^{n-1}(\Gamma,\Zp\otimes\lambda)\otimes_{\Zp}\Zp[\Gamma\backslash(L/p^rL)].
\end{equation*}
Composing this homomorphism with the projection
\[
H^{n-1}(\Gamma,\Zp[[L_{\Zp}]]\otimes \lambda)\to
H^{n-1}(\Gamma,\Zp[L/p^rL]\otimes \lambda )
\]
and passing to the projective limit give the desired homomorphism.
\end{proof}
Using this proposition the Eisenstein distribution gives rise
to an element in
\begin{equation}
\Meas(L_{\Zp}/\Gamma, \Meas(\Gamma\backslash L_{\Zp}, H^{n-1}(\Gamma,\Zp\otimes \lambda)))\isom
\Meas(\Gamma\backslash L_{\Zp}\times \Gamma\backslash L_{\Zp},H^{n-1}(\Gamma,\Zp\otimes \lambda))
\end{equation}
where the isomorphism comes from the Fubini theorem about
integration on a product space. Note that in the case where 
$\Gamma\subset \Aut(L)$ is an arithmetic subgroup, the $\Zp$-module
$H^{n-1}(\Gamma,\Zp\otimes \lambda)$ is finitely generated, so that
the Eisenstein distribution becomes a measure on $\Gamma\backslash L_{\Zp}\times \Gamma\backslash L_{\Zp}$.
 
The next theorem 
shows that the Eisenstein distribution does not give anything
essentially new.
\begin{theorem}For any $g\in (\Zp[T^{(p)}\smallsetminus \{0\}]^0)^{\Gamma}$ the Eisenstein measure 
\[
\mu_{L,\Eis}^g\in\Meas(\Gamma\backslash L_{\Zp}\times \Gamma\backslash L_{\Zp},H^{n-1}(\Gamma,\Zp\otimes \lambda))
\]
is supported on the diagonal $\Gamma\backslash L_{\Zp}\xrightarrow{\Delta}\Gamma\backslash L_{\Zp}\times \Gamma\backslash L_{\Zp}$.
\end{theorem}
\begin{proof}
The proof is formal and one has just to unravel the definition
of $\mu_{L,\Eis}^g$. It certainly suffices to show this
for $\mu^g_{r,L,\Eis}$, i.e., to work with $L_r:=L/p^rL$ and $A_r:=\ZZ/p^r\ZZ$. 
Let $L':=p^rL$ and
$t\in L$. Then $t$ defines a $p^r$-torsion point on $T'$
because this group is just $L_r$. We assume in this proof that
$\Gamma$ acts trivially on $L_r$ otherwise one has to replace
$t$ by some linear combination of $p^r$-torsion sections.
We consider the Eisenstein measure 
\[
\mu^g_{r,L,\Eis}:A_r[L_r]^\Gamma\to 
H^{n-1}(\Gamma, A_r[L_r]\otimes \lambda)
\]
as a map from the $\Gamma$-invariant functions on $L_r$ to 
$H^{n-1}(\Gamma, A_r[L_r]\otimes \lambda)$. By its construction $\mu^g_{r,L,\Eis}$
evaluated at $\delta_t\in A_r[L_r]$ is given by 
$\Tr_{[p^r]}(\Eis(\delta_t\otimes [p^r]^*(g)))$, where $[p^r]$ is the isogeny $[p^r]:L'\to L$. We have
\[
\Eis(\delta_t\otimes [p^r]^*(g))
\in H^{n-1}(\Gamma, t^*\sLog'\otimes \lambda').
\]
We claim that $t^*\sLog'\isom \Meas(t+L'_{\Zp},A_r)$. This follows
because $\sLog'\isom \pi_!A_r[L']\otimes_{A_r}R'$ by the construction
of $\sLog'$ and because $t^*\pi_!A_r[L']=A_r[t+L']$, where
we denote by $A_r[t+L']$ the free $A_r$-module on $t+L'$. Taking
the tensor product with $R$ completes this $\Zp[L']$-module and
gives the desired isomorphism.

The isogeny $[p^r]$ induces a map $[p^r]_\sLog:t^*\sLog'\to 0^*\sLog$ which identifies with 
\[
[p^r]_*:\Meas(t+L'_{\Zp},A_r)\to \Meas(L_{\Zp},A_r).
\]
If we compose this with $\Meas(L_{\Zp},A_r)\to \Meas(L_r,A_r)=A_r[L_r]$ we see that the image of 
$\Meas(t+L'_{\Zp},A_r)$ in $A_r[L_r]$ is given by $A_r\delta_t$.
In particular, if we consider the pairing
\[
A_r[L_r]^\Gamma\otimes H^{n-1}(\Gamma, A_r[L_r]\otimes \lambda)
\to H^{n-1}(\Gamma, A_r\otimes\lambda)
\]
we see that $\delta_s\otimes \mu^g_{r,L,\Eis}(\delta_t)\mapsto 0$
for $s\neq t$. If we rewrite the Eisenstein measure 
as 
\[
\mu^g_{r,L,\Eis}:A_r[L_r]^\Gamma\times A_r[L_r]^\Gamma
\to H^{n-1}(\Gamma, A_r\otimes \lambda)
\]
this just means that $\mu^g_{r,L,\Eis}(\delta_t\otimes\delta_{s})=0$ for $t\neq s$,
i.e., that $\mu^g_{r,L,\Eis}$ is supported on the diagonal.
\end{proof}
\begin{corollary}
The Eisenstein measure 
$\mu_{L,\Eis}^g\in \Meas(\Gamma\backslash L_{\Zp},H^{n-1}(\Gamma,A[[L_{\Zp}]]\otimes \lambda))$
is completely determined by its image 
\[
\overline{\mu}_{L,\Eis}^g\in \Meas(\Gamma\backslash L_{\Zp},H^{n-1}(\Gamma,A\otimes \lambda))
\]
under the augmentation map $A[[L_{\Zp}]]\to A$. It is
also completely determined by its value on the constant function
$1$ on $L_{\Zp}/\Gamma$:
\[
\mu_{L,\Eis}^g(1)\in H^{n-1}(\Gamma,A[[L_{\Zp}]]\otimes \lambda).
\]
\end{corollary}
\begin{proof}
This is just a reformulation of the theorem.
\end{proof}
In the case of totally real fields, this has the following
consequence. 
\begin{corollary}
Let $F$ be a totally real field with ring of integers $\cO_F$
and $\Gamma\subset \cO_F^{+,\times}$ a subgroup of finite index
of the totally positive units of $\cO_F$. Let $L=\frf\frb^{-1}$
be a fractional ideal. Then the
measure 
\[
\ev(\Eis_{\alpha_{[\frc]}})\in (R_{\Zp})_\Gamma
\]
from Corollary \ref{cor:p-adic-interpolation} coincides
with the measure $\overline{\mu}^{\alpha_{[\frc]}}_{L,\Eis}$.
\end{corollary}

%
\subsection{Relation with Eisenstein cohomology on Hilbert modular
varieties}
%
In this section we explain a construction of 
Graf which uses the topological polylogarithm
to get the Eisenstein cohomology on Hilbert modular varieties in all cohomological degrees. For more details and the actual comparison 
with Harder's Eisenstein cohomology we refer to the forthcoming thesis
of Graf \cite{Graf}. We present here a slight variant
of his results and we 
will deduce from our general principles 
some integrality and $p$-adic interpolation properties of the
Eisenstein cohomology classes.

Let again $F/\QQ$ be a totally real field of degree $n$ with
ring of integers $\cO_F$. We define for any fractional
ideal $\fra$ of $F$ the group
\[
\GL_2^+(\cO_F,\fra):=\{\begin{pmatrix} a & b\\ c& d\end{pmatrix}\in \GL_2(F)\mid  a, d\in\cO_F, b\in \fra, c\in \fra^{-1}, ad-bc\in \cO_F^{+,\times}\}.
\]
We identify the centre of $\GL_2^+(\cO_F,\fra)$ 
with $\cO_F^\times$ 
and we let $\PGL_2(\cO_F,\fra)$ be the quotient.
The group $\GL_2^+(\cO_F,\fra)$ acts on $(F\otimes \RR)^2\isom F\otimes \CC$ from the right and stabilizes the lattice
\[
L:=\cO_F \cdot 1+\fra\cdot \sqrt{-1}\subset F\otimes \CC.
\]
We consider the torus $T:=F\otimes \CC/L$ of real dimension
$2n$ and an integer $N>1$ which is invertible in $A$. 
We let 
\begin{equation}
D:=T[N]\setminus\{0\}\subset T
\end{equation}
be the $N$-torsion subgroup without the $0$-section and
denote by  $\Gamma\subset\GL_2^+(\cO_F,\fra)$
the stabilizer of $D$. Let $\Delta:=\Gamma\cap \cO_F^\times$ be the intersection of $\Gamma$ with the centre. Then
$\Gamma\subset \GL_2^+(\cO_F,\fra)$ and 
$\Delta\subset\cO_F^\times $ are subgroups of finite index and we 
define
$\Gamma':=\Gamma/\Delta\subset \PGL_2(\cO_F,\fra)$, so that 
we have an exact sequence
\[
0\to\Delta\to \Gamma\to \Gamma'\to 0.
\]
\begin{remark}
To have a geometric perspective on this, we define $\GL_2^+(F\otimes\RR):=\{(\omega_1,\omega_2)\in (F\otimes \CC)^2\mid \Im(\frac{\omega_2}{\omega_1})>0\}$. 
Then $\left(\begin{smallmatrix} a & b\\ c& d\end{smallmatrix}\right)\in \Gamma$
acts on $(\omega_1,\omega_2)\in\GL_2^+(F\otimes\RR)$ by right
multiplication $(\omega_1,\omega_2)\left(\begin{smallmatrix} a & b\\ c& d\end{smallmatrix}\right)$ and $\lambda\in (F\otimes \CC)^\times$ acts
by left multiplication $\lambda(\omega_1,\omega_2)=(\lambda\omega_1,\lambda\omega_2)$.
The map $(\omega_1,\omega_2)\mapsto \tau:=\frac{\omega_2}{\omega_1}$
identifies the quotient $(F\otimes \CC)^\times\backslash \GL_2^+(F\otimes \RR)$ with the upper half plane 
$F\otimes\HH:=\{\tau\in F\otimes \CC\mid \Im\tau \mbox{ totally positive}\}$. Note that the map 
$\GL_2^+(F\otimes \RR)\to F\otimes\HH$ is compatible with
the homomorphism $\Gamma\to \Gamma'$.
\end{remark}

The  Eisenstein class of Definition \ref{def:Eisenstein-special-case}
provides us for any ring $A$ in which $N$ is invertible with a map
\begin{equation}\label{eq:Eis-map}
\Eis:(A[D]^0)^\Gamma\to H^{2n-1}(\Gamma, R\otimes\lambda). 
\end{equation}
We explain how to get cohomology classes in other degrees
starting from this class. Consider the Hochschild-Serre spectral sequence
\[
H^{2n-1-p}(\Gamma', H^p(\Delta, R\otimes \lambda))\Rightarrow
H^{2n-1}(\Gamma, R\otimes\lambda). 
\]
As $\Delta$ has cohomological dimension $n-1$ we have an edge morphism
\begin{equation}
H^{2n-1}(\Gamma, R\otimes\lambda)\to 
H^n(\Gamma',H^{n-1}(\Delta,R\otimes \lambda)). 
\end{equation}
If we compose this with the cap-product with
$H_{n-1}(\Delta,\ZZ)$ we get a map
\[
\Eis:(A[D]^0)^\Gamma\otimes H_{n-1}(\Delta,\ZZ)\to H^{n}(\Gamma', (R\otimes\lambda)_\Delta).
\]
In order to get also the Eisenstein classes in other cohomological
degrees consider $\cO_N^{+,\times}:=\{u\in \cO_F^{+,\times}\mid
u\equiv 1\bmod N\}$ and the determinant map $\Gamma\xrightarrow{\det}\cO^{+,\times}_N$. For the rest of the
section we use the following notation:
\begin{equation}
\Lambda^\cdot:=\Lambda^\cdot\Hom(\cO^{+,\times}_N, \ZZ)= H^\cdot(\cO^{+,\times}_N,\ZZ).
\end{equation}
Then the map $\det$ gives rise to a ring homomorphism
\[
{\det}^*: \Lambda^\cdot\to H^\cdot(\Gamma,\ZZ)
\]
so that $H^\cdot(\Gamma, \ZZ)$ becomes a $\Lambda^\cdot$-module. 
Therefore \eqref{eq:Eis-map} yields the map 
\[
\Eis:(A[D]^0)^\Gamma\otimes \Lambda^p\to H^{2n-1+p}(\Gamma, R\otimes\lambda).
\]
A further composition with the edge morphism and the cap-product
with $H_{n-1}(\Delta,\ZZ)$ gives:
\begin{definition}
For each $0\le p\le n-1$ we define the Eisenstein cohomology
operator in degree $n+p$ to be the map
\begin{equation*}
\Eis:A[D]^0\otimes \Lambda^p\otimes H_{n-1}(\Delta,\ZZ)\to H^{n+p}(\Gamma', (R\otimes\lambda)_\Delta).
\end{equation*}
Composing with $\exp^*_k$ gives
\[
\Eis^k:A[D]^0\otimes \Lambda^p\otimes H_{n-1}(\Delta,\ZZ)\to H^{n+p}(\Gamma', (\TSym^kL_A)_\Delta\otimes\lambda).
\]
\end{definition}
\begin{remark}
For $A$ a $\QQ$-algebra, one can 
show that $(\TSym^kL_A)_\Delta=0$ if
$k$ is not a multiple of $n$ and non-trivial otherwise.
\end{remark}
Choose generators
for $ \Lambda^p$ and $H_{n-1}(\Delta,\ZZ)\isom {\Lambda}^{n-1}\Delta$, then
we get directly from the construction the following integrality result
for the Eisenstein cohomology:
\begin{proposition}
Let $\alpha\in \ZZ[\frac{1}{N}][D]^0$ 
then with the above generators
\[
\Eis^k_\alpha\in H^{n+p}(\Gamma', (\TSym^kL_{\ZZ[\frac{1}{N}]})_\Delta\otimes\lambda).
\]
\end{proposition}
Keeping the generators and putting $A=\Zp$ we get
also a $p$-adic interpolation result. Recall \ref{prop:R-as-Iwasawa} that
for $A=\Zp$ one has an isomorphism $A[[L_{\Zp}]]\isom R$.
\begin{proposition}
With the above notations, for each  $\alpha\in \Zp[D]^0$ 
the class
\[
\Eis_\alpha\in H^{n+p}(\Gamma', (A[[L_{\Zp}]])_\Delta\otimes\lambda)
\]
has the interpolation property that 
$\mom^k(\Eis_\alpha)=\Eis^k_\alpha\in H^{n+p}(\Gamma', (\TSym^kL_{\Zp})_\Delta\otimes\lambda)$.
\end{proposition}
\begin{proof}
This is clear from the construction.
\end{proof}
\begin{proposition}[Graf \cite{Graf}]
If $A$ is a $\QQ$-algebra then the product map 
\[
\bigoplus_p H^{2n-1-p}(\Gamma',R^\Delta\otimes \lambda)\otimes\Lambda^p\to 
H^{2n-1}(\Gamma, R\otimes \lambda)
\]
is an isomorphism.
\end{proposition}
\begin{proof}
One has $\Lambda^\cdot\otimes\QQ\isom H^\cdot(\Delta,\QQ)$. As in
Remark \ref{rem:coinvariants} the projection $p_\Delta:R\to R_\Delta$
then gives rise to  isomorphisms
\[
H^p(\Delta, R\otimes\lambda)\isom H^p(\Delta,R_\Delta\otimes\lambda)\isom
R_\Delta\otimes\lambda\otimes\Lambda^p
\]
and the result follows from the Hochschild-Serre spectral sequence.
\end{proof}
\begin{remark}
In his thesis Graf decomposes the topological polylogarithm according to the isomorphism in the above
proposition and shows how  the
resulting cohomology classes are related with the ones constructed
by Harder. For this he explicitly computes the residue of this
classes at the boundary. 
\end{remark}

\bibliographystyle{amsalpha}
\bibliography{integral-polylog}
\end{document}